\documentclass[11pt,twoside]{article}
\usepackage{times}
\usepackage{amsmath,amssymb,amsthm}
\usepackage{enumerate}
\usepackage{cite}
\allowdisplaybreaks

\newcommand{\no}{\nonumber}

\newcommand{\R}{\mathbb R}
\newcommand{\N}{\mathbb N}
\newcommand{\Z}{\mathbb Z}
\newcommand{\p}{\partial}
\newcommand{\ve}{\varepsilon}
\newcommand{\f}{\frac}

\newcommand{\al}{\alpha}
\renewcommand{\t}{\tilde}

\renewcommand{\th}{\theta}

\newcommand{\ds}{\displaystyle}
\newcommand{\md}{\mathrm{d}}

\allowdisplaybreaks

\newcommand{\crit}{\textup{crit}}
\newcommand{\conf}{\textup{conf}}

\theoremstyle{plain}
\newtheorem{theorem}{Theorem}[section]

\newtheorem{lemma}[theorem]{Lemma}

\newtheorem*{claim}{Claim}

\theoremstyle{definition}

\theoremstyle{remark}
\newtheorem{remark}{Remark}[section]

\numberwithin{equation}{section}

\title{On semilinear Tricomi equations with critical exponents
or in two space dimensions}
  \author{Daoyin He$^{1*}$,\qquad Ingo Witt$^{2,*}$, \qquad Huicheng
  Yin$^{3,}$\footnote{He Daoyin (\texttt{daoyinhe@mathematik.uni-goettingen.de})
    and Yin Huicheng
    (\texttt{huicheng@} \texttt{nju.edu.cn}) are supported by the NSFC
    (No.~11571177) and by the Priority Academic Program Development of
    Jiangsu Higher Education Institutions. Ingo~Witt (\texttt{iwitt@mathematik.uni-goettingen.de})
     was partly supported
    by the DFG through the Sino-German project ``Analysis of PDEs and
    Applications." }\vspace{0.5cm}\\ \small 1.
  Department of Mathematics and IMS, Nanjing University, Nanjing
  210093, China.\\ \small 2.  Mathematical Institute, University of
  G\"{o}ttingen, Bunsenstr.~3-5, D-37073 G\"{o}ttingen,
  Germany.\\ \small 3.  School of Mathematical Sciences, Jiangsu Provincial Key Laboratory for Numerical Simulation\\
\small of Large Scale Complex Systems, Nanjing Normal University, Nanjing 210023, China.\\}
\vspace{0.5cm}
\begin{document}
\date{}

\maketitle
\thispagestyle{empty}

\begin{abstract}
This paper is a complement of our recent works on  the semilinear Tricomi equations
in \cite{HWYin1} and \cite{HWYin2}.
For the semilinear Tricomi equation $\partial_t^2 u-t\Delta
u=|u|^p$ with initial data $(u(0,\cdot), \p_t u(0,\cdot))$ $=(u_0, u_1)$,
where $t\ge 0$, $x\in\R^n$ ($n\ge 3$), $p>1$, and $u_i\in
C_0^{\infty}(\R^n)$ ($i=0,1$),  we have shown in  \cite{HWYin1} and \cite{HWYin2}
that there exists a critical
exponent $p_{\crit}(n)>1$
such that the solution $u$, in general,
blows up in finite time when $1<p<p_{\crit}(n)$, and there is a global small solution  for \(p>p_{\crit}(n)\).
In the present paper,
firstly, we prove that the solution of $\partial_t^2 u-t\Delta
u=|u|^p$ will generally blow up for the critical exponent \(p=p_{\crit}(n)\) and \(n\geq2\),
secondly, we establish the global existence of small data solution
to $\partial_t^2 u-t\Delta
u=|u|^p$  for $p>p_{crit}(n)$ and $n=2$. Thus,
we have given a systematic study on the blowup or global existence of small data solution $u$ to
the equation $\partial_t^2 u-t\Delta u=|u|^p$  for $n\ge 2$.

\end{abstract}

\vskip 0.2 true cm


\vskip 0.2 true cm


\vskip 0.4 true cm

\section{Introduction}

In this paper, we continue to be concerned with the global existence or blowup
of solutions $u$ to the semilinear Tricomi equation
\begin{equation}\label{equ:original}
\left\{ \enspace
\begin{aligned}
&\partial_t^2 u-t\Delta u =|u|^p, \\
&u(0,\cdot)=f(x), \quad \partial_{t} u(0,\cdot)=g(x),
\end{aligned}
\right.
\end{equation}
where $t\geq 0$, $x=(x_1,\dots, x_n)\in\R^n$ ($n\ge 2$),
$p>1$, and $u_i\in C_0^{\infty}(B(0,M))$ $(i=0,1)$ with $B(0, M)=\{x: |x|=\sqrt{x_1^2+...+x_n^2}<M\}$ and $M>1$.
For the local well-posedness and optimal regularities of solution \(u\) to problem \eqref{equ:original},
the readers may consult \cite{Rua1,Rua2,Rua3,Rua4,Yag2} and the references therein.
In \cite{HWYin1}-\cite{HWYin2}, we have determined a critical exponent $p_{crit}(n)$ and a conformal
exponent $p_{conf}(n)$ $(>p_{crit}(n)\big)$ for \eqref{equ:original} as follows (corresponding to the case of $m=1$
in the generalized equation Tricomi equation $\partial_t^2 u-t^m\Delta u =|u|^p$):
$p_{crit}(n)$ is the positive root of the algebraic equation
\begin{equation}\label{equ:1.1}
(3n-2)p^2-3np-6=0,
\end{equation}
and $p_{conf}(n)=\frac{3n+6}{3n-2}$.
It is shown in \cite{HWYin1} that for all \(n\geq2\),
the solution $u$ of \eqref{equ:original} generally blows up in
finite time when $1 < p < p_{crit}(n)$, and
meanwhile $u$ exists globally when $p\ge p_{conf}(n)$ for small initial data and $n\ge 2$. In \cite{HWYin2},
we prove that the small data solution $u$ of \eqref{equ:original} exists globally when \(n\geq3\) and $p_{crit}(n)<p< p_{conf}(n)$.
Therefore, collecting the results in \cite{HWYin1}-\cite{HWYin2},
we have given a detailed study on the blowup or global existence of small data solution $u$ to
problem \eqref{equ:original} for $n\ge 3$ except $p=p_{crit}(n)$, and  for $n=2$
with $p\ge p_{conf}(n)$ except $p_{crit}(n)<p< p_{conf}(n)$.
In this paper, firstly, we establish the finite time blowup result for problem \eqref{equ:original}
when $n\ge 2$ and $p=p_{crit}(n)$, secondly, we prove the global existence of small data solution $u$ to
problem \eqref{equ:original} when $n=2$ and $p_{crit}(n)<p< p_{conf}(n)$.

\begin{theorem}\label{thm:1.1}
Let \(n\geq2\) and \(p=p_{crit}(n)\). Suppose that the initial data \(f,g\in C^{\infty}_0(\R^n)\) are non-negative and
positive somewhere, then  problem \eqref{equ:original} admits no global solution \(u\) with
\[u\in C^1\big([0,\infty), H^1(\R^n)\big)\cap C\big([0,\infty), L^2(\R^n)\big).\]
\end{theorem}

\begin{theorem}\label{thm:1.2}
Let $n=2$.
For $p_{crit}(n)<p\le p_{conf}(n)$, suppose that the initial data \((f,g)\) satisfy
\begin{equation}\label{equ:1.2}
\sum_{|\alpha|\leq1}(\|Z^\alpha f\|_{\dot{H}^s(\R^2)}+\|Z^\alpha g\|_{\dot{H}^{s-{\f{2}{3}}}(\R^2)})<\varepsilon,
\end{equation}
where $\varepsilon>0$ is a sufficiently small constant, $s=1-\f{4}{3(p-1)}$, and $Z=\{\partial_1, \partial_2, x_1\partial_2-x_2\partial_1\}$.
Then problem \eqref{equ:original} admits a global solution \(u\) such that
\[u\in L^q_tL^p_{r}L_\theta^2(\R_+^{1+2}),\quad \text{i.e., $\biggl(\int_0^{\infty} \bigl(\int_0^{\infty}\bigl(\int_0^{2\pi}|u(t, r cos\th, r sin\th)|^2d\th\bigr)^{\f{p}{2}}dr\bigr)^{\f{q}{p}}dt\biggr)^{\f{1}{q}}<\infty$},\]
where $x_1=r cos\th$ and $x_2=r sin\th$ with $r\ge 0$ and $\th\in [0, 2\pi]$, \(q=\f{p(p-1)}{3-p}>2\)
for $p_{crit}(2)<p\le\f73$; \(q=\f{(3p+1)(p-11)}{11-3p}>2\)
for $\f73<p\leq\f{4+\sqrt{17}}{3}$;
\(q=\f{p^2-1}{5-p}>2\)
for $\f{4+\sqrt{17}}{3}<p\le p_{conf}(2)=3$.
\end{theorem}

\begin{remark}
For the semilinear wave equation $\partial_t^2 u-\Delta u =|u|^p$ $(p>1)$,
the critical exponent $p_0(n)$ in Strauss' conjecture (see \cite{Strauss})
is determined by the algebraic equation $(n-1)p_0^2(n)-(n+1)p_0(n)-2=0$
(so far the global existence of small data solution $u$  for $p>p_0(n)$  or the blowup
of solution $u$ for $1<p<p_0(n)$  have been proved in \cite{Gla1}-\cite{Gls}, \cite{Joh}-\cite{Ls}, \cite{Sid} and the references therein).
The finite time blowup  for the critical wave equations $\partial_t^2 u-\Delta u =|u|^{p_0(n)}$
has been established in \cite{Gla1},\cite{Joh}, \cite{Sch}, and \cite{Yor}-\cite{Zhou}, respectively.
Motivated by the techniques in \cite{Yor} and \cite{HWYin1}, we prove the blowup result
for the critical semilinear Tricomi equation in \eqref{equ:original}.
\end{remark}

\begin{remark}
For brevity, we only study the semilinear Tricomi equation instead of the
generalized semilinear Tricomi equation $\partial_t^2 u-t^m\Delta u =|u|^p$
$(m\in\Bbb N)$ in problem \eqref{equ:original}. In fact, by the methods in Theorem 1.1
and Theorem 1.2, one can establish the analogous results to Theorem 1.1-Theorem 1.2 for the
generalized semilinear Tricomi equation.
\end{remark}

\begin{remark}
It follows from a direct computation that $p_{crit}(2)=\f{3+\sqrt{33}}{4}$
and $p_{conf}(2)=3$ in Theorem 1.2.
\end{remark}

For $n=1$, the linear equation $\p_t^2u-t\p_x^2u=0$ is the well-known Tricomi equation
which arises in transonic gas dynamics.
There are extensive results
for both linear and semilinear Tricomi equations in $n$ space dimensions
$(n\in\Bbb N)$. For
instances, with respect to the linear Tricomi equation,
the authors in \cite{Bar}, \cite{Yag1} and \cite{Yag3} have computed its
fundamental solution explicitly; with respect to the semilinear Tricomi equation $\p_t^2u-t\Delta u=f(t,x,u)$,
under some certain assumptions on the function
$f(t,x,u)$, the authors in \cite{Gva} and \cite{Lup1}-\cite{Lup4}
have obtained a series of interesting  results on the existence and uniqueness of solution $u$
in bounded domains under Tricomi, Goursat or Dirichlet boundary conditions respectively
in the mixed type case, in the degenerate hyperbolic setting or in the degenerate
elliptic setting; with respect to the  Cauchy problem of semilinear Tricomi equations,
the authors in \cite{Rua1, Rua2, Rua3, Rua4} established the local existence as well
as the singularity structure of low regularity solutions in the
degenerate hyperbolic region and the elliptic-hyperbolic mixed region,
respectively. In addition, by establishing some classes of $L^p$-$L^q$  estimates
for the solution $v$ of linear equation $\partial_t^2v -t\triangle v=F(t,x)$,
the author in \cite{Yag2} obtained some results about the global existence or the
blowup of solutions to problem (1.1) when the exponent
$p$ belongs to a certain range, however, there was a
gap between the global existence interval and the blowup interval.
By establishing the Strichartz inequality and the weighted Strichartz inequality
for the linear Tricomi equation, respectively, we have shown the global existence of small data solution $u$ to
problem \eqref{equ:original} for $p>p_{crit}(n)$ ($n\ge 3$) in  \cite{HWYin1}- \cite{HWYin2}.

We now comment on the proof of Theorem~\ref{thm:1.1} and Theorem~\ref{thm:1.2}.
To prove Theorem~\ref{thm:1.1}, we define the
function $G(t)=\int_{\R^n}u(t,x)\ \md x$. By
applying some crucial techniques for the modified Bessel function as
in \cite{Hon, Rua3}, and motivated by \cite{Yor} and \cite{HWYin1},  we can derive a
Riccati-type ordinary differential inequality for $G(t)$ through a delicate
analysis of \eqref{equ:original}, which is stronger than the ordinary differential
inequality in \cite{HWYin1} (see (2.1) of \cite{HWYin1}). From this and Lemma 2.1 in \cite{Yor},
the blowup result for \(p=p_{crit}(n)\) in Theorem 1.1 is established under the positivity assumptions of $f$ and $g$.
To prove the global existence result in
Theorem~\ref{thm:1.2}, we require to establish angular Strichartz estimates for the
Tricomi operator $\p_t^2-t\Delta$ as in the treatment on the 2-D linear wave operator in \cite{Smi}.
In this process, a series of
inequalities are derived by applying an explicit formula for the solutions
of linear Tricomi equations and by utilizing some
basic properties of related Fourier integral operators and some classical results in harmonic analysis.
Based on the resulting Strichartz inequalities and the contractible mapping principle,
we eventually complete the proof of Theorem~\ref{thm:1.2}.
Here we point out that compared with the techniques of \cite{Smi} for deriving the Strichartz inequality with
angular mixed-norm of 2-D linear wave equation, due to the influences of degeneracy
and variable coefficients in the linear equation, it is more involved
and complicated to give the related analysis on the resulting Fourier integral operator from
linear Tricomi equation.

This paper is organized as follows: In Section 2,
we complete the proof of Theorem 1.1. In Section 3,
some Strichartz estimates with angular mixed norms for the linear Tricomi equation
are established. In Section 4, by applying the results in Section 3, Theorem~\ref{thm:1.2} is proved.

\section{Proof of Theorem~\ref{thm:1.1}}
Before starting the proof of Theorem~\ref{thm:1.1}, we cite a blowup lemma from \cite{Yor}.
\begin{lemma}\label{lem:2.1}
Let \(p>1\), \(a\geq1\), and \((p-1)a=q-2\). Suppose \(G\in C^2[0,T)\) satisfies that for \(t\geq T_0>0\),
\begin{align}
G(t)\geq & K_0(t+M)^a, \label{equ:2.01}\\
G''(t)\geq & K_1(t+M)^{-q}G(t)^p, \label{equ:2.02}
\end{align}
where \(K_0\), \(K_1\), \(T_0\) and \(M\) are some positive constants. Fixing \(K_1\), there exists a positive constant \(c_0\), independent of \(M\) and \(T_0\) such that if \(K_0\geq c_0\), then \(T<\infty\).
\end{lemma}
With this lemma and $G(t)=\int_{\R^n}u(t,x)\ \md x$, our subsequent tasks are to derive \eqref{equ:2.01} and \eqref{equ:2.02}
for  the solution $u$ of problem \eqref{equ:original}. It follows from Section 2 of \cite{HWYin1} that
\begin{align}\label{1-0}
G''(t)=\int_{\R^n}|u(t,x)|^p \md x\geq C(M+t)^{-\frac{3}{2}\,n(p-1)}\,|G(t)|^p.
\end{align}
This means that \eqref{equ:2.02} holds for \(q=\frac{3}{2}\,n(p-1)\). Next, strongly motivated by the techniques in \cite{Yor} and \cite{HWYin1},
we focus on the derivation of \eqref{equ:2.01},
which is divided into the following three steps:

\vskip 0.2 true cm

\textbf{Step 1. Some reductions}

\vskip 0.2 true cm

Let
\[\bar{u}=\f{1}{\omega_n}\int_{\mathbb{S}^{n-1}}u(t,r\theta)\md \theta\]
be the spherical average of \(u\). Then applying the spherical average on both sides of \eqref{equ:original} yields
\begin{equation}\label{A-0}
\partial_t^2 \bar{u}-t \overline{\Delta u}=\overline{|u|^p}.
\end{equation}
By Daboux's identity, one has \(\overline{\Delta u}=\Delta\bar{u}\). On the other hand, it follows from H\"{o}lder's inequality
that
\[|\bar{u}|=\bigg|\f{1}{\omega_n}\int_{\mathbb{S}^{n-1}}u(t,r\theta)\md \theta\bigg|\leq \bigg(\f{1}{\omega_n}\int_{\mathbb{S}^{n-1}}|u|^p\md \theta\bigg)^{\f{1}{p}}\bigg(\int_{\mathbb{S}^{n-1}}\f{\md\theta}{\omega_n}\bigg)^{\f{1}{p'}}\leq\big(\overline{|u|^p}\big)^{\f{1}{p}}.\]
This, together with \eqref{A-0}, yields
\begin{equation}\label{equ:2.1}
\partial_t^2 \bar{u}-t \Delta\bar{u}\geq|\bar{u}|^p.
\end{equation}
Thus we can assume that \(u\) is radial since the blowup of \(\bar{u}\) obviously
yields the blowup of \(u\).
Let \(\omega\in\R^n\) be a unit vector. The Radon transform of \(u\) with respect to the variable \(x\) is defined as
\begin{equation}\label{equ:2.2}
\mathbf{R}(u)(t,\rho)=\int_{x\cdot\omega=\rho}u(t,x)\md S_x,
\end{equation}
where $\rho\in\Bbb R$, \(\md S_x\) is the Lebesque measure on the hyper-plane \(\{x: x\cdot\omega=\rho\}\).
From \eqref{equ:2.2} and the radial assumption of \(u(t,\cdot)\), it is easy to see
\begin{equation}\label{equ:2.3}
\begin{split}
\mathbf{R}(u)(t,\rho) & = \int_{\{x': x'\cdot\omega=0\}}u(t,\rho\omega+x')\md S_x \\
& =c_n\int_{|\rho|}^{\infty}u(t,r)(r^2-\rho^2)^{\f{n-3}{2}}r\md r.
\end{split}
\end{equation}
Obviously, \(\mathbf{R}(u)(t,\rho)\) is independent of \(\omega\).

\vskip 0.2 true cm

\textbf{Step 2. The lower bound of \(\mathbf{R}(u)\)}

\vskip 0.2 true cm

From Page 3 of \cite{Hel}, we have 
\begin{equation}\label{equ:2.4}
  %\begin{split}
  \mathbf{R}(\Delta u)(t,\rho) =\partial_\rho^2\mathbf{R}(u)(t,\rho).
  %\int_{\{x'\mid x'\cdot\omega=0\}}\Delta u(t, \rho\omega+x')\md S_x \\
  %& =\int_{\{x'\mid x'\cdot\omega=0\}}\sum_i\omega_i^2\partial_\rho^2u(t, \rho\omega+x')\md S_x \\
  %& =\int_{\{x'\mid x'\cdot\omega=0\}}\partial_\rho^2u(t, \rho\omega+x')\md S_x \\
  %& =\partial_\rho^2\mathbf{R}(u)(t,\rho)
  %\end{split}
\end{equation}
Since \(u\) is a solution of \eqref{equ:original}, it follows from \eqref{equ:2.4} that \(\mathbf{R}(u)\) solves
\begin{equation*}
  \left\{ \enspace
\begin{aligned}
&\partial_t^2 \mathbf{R}(u)-t\partial_\rho^2\mathbf{R}(u) =\mathbf{R}(|u|^p), \quad (t, \rho)\in \R^{1+1}_{+},\\
&\mathbf{R}(u)(0,\rho)=\mathbf{R}(f), \quad \partial_{t} \mathbf{R}(u)(0,\rho)=\mathbf{R}(g).
\end{aligned}
\right.
\end{equation*}
By Lemma 2.1 in \cite{Yag3} and Theorem 3.1 in \cite{Yag1}, we have
\begin{equation}\label{equ:2.5}
\begin{split}
\mathbf{R}(u)&(t,\rho)  =C\int_{0}^{1}v_{\mathbf{R}(f)}\big(\phi(t)s,\rho\big)(1-s^2)^{-\f{5}{6}}\md s+C t\int_{0}^{1}
v_{\mathbf{R}(g)}\big(\phi(t)s,\rho\big)(1-s^2)^{-\f{1}{6}}\md s \\
& +C\int_0^t\int_{\rho-\phi(t)+\phi(s)}^{\rho+\phi(t)-\phi(s)}\big(\phi(t)+\phi(s)+\rho-\rho_1\big)^{-\gamma}
\big(\phi(t)+\phi(s)-\rho+\rho_1\big)^{-\gamma}F\big(\f{1}{6},\f{1}{6},1,z\big)\\
&\qquad\qquad \qquad \qquad \quad \times \mathbf{R}(|u|^p)(s,\rho_1)\md\rho_1\md s,
\end{split}
\end{equation}
where $C>0$ is a constant, $z=\frac{(\rho-\rho_1+\phi(t)-\phi(s))(\rho-\rho_1-\phi(t)+\phi(s))}{(\rho-\rho_1+\phi(t)+\phi(s))(\rho-\rho_1-\phi(t)-\phi(s))}$
with $\phi(t)=\f{2}{3}t^{\f{3}{2}}$, $F\big(\f{1}{6},\f{1}{6},1,z\big)$ is the hypergeometric function, and the function \(v_\varphi\)
solves the 1-D wave equation
\begin{equation*}
  \left\{ \enspace
\begin{aligned}
&\partial_t^2 v-\partial_x^2v =0, \quad (t, x)\in \R^{1+1}_{+},\\
&v(0,x)=\varphi, \quad \partial_{t} v(0,x)=0.
\end{aligned}
\right.
\end{equation*}
Note that \eqref{equ:2.3} together with the non-negativity of $f$ and $g$ shows \(\mathbf{R}(f)\geq0\) and \(\mathbf{R}(g)\geq0\).
In addition, by D'Alembert's formula, we obtain \(v_{\mathbf{R}(f)}\geq0\) and \(v_{\mathbf{R}(g)}\geq0\).
Hence,
\begin{align*}
\mathbf{R}(u)(t,\rho) & \geq C_\gamma\int_0^t\int_{\rho-\phi(t)+\phi(s)}^{\rho+\phi(t)-\phi(s)}\big(\phi(t)+\phi(s)+\rho-\rho_1\big)^{-\f{1}{6}}
\big(\phi(t)+\phi(s)-\rho+\rho_1\big)^{-\f{1}{6}} \\
& \qquad\qquad\qquad \qquad\quad  \times F\big(\f{1}{6},\f{1}{6},1,z\big)\mathbf{R}(|u|^p)(s,\rho_1)\md\rho_1\md s.
\end{align*}
Note that
\[z=\frac{(\phi(t)-\phi(s))^2-(\rho-\rho_1)^2}{(\phi(t)+\phi(s))^2-(\rho-\rho_1)^2}\in[0,1].\]
Then by page 59 of \cite{Erd1}, we arrive at
\begin{equation*}
\begin{split}
F(\f{1}{6},\f{1}{6},1,z)=&\frac{1}{\Gamma(\f{1}{6})\Gamma(\f{5}{6})}\int_0^1t^{-\f{5}{6}}(1-t)^{-\f{1}{6}}
(1-zt)^{-\f{1}{6}}dt \\
\geq &\frac{1}{\Gamma(\f{1}{6})\Gamma(\f{5}{6})}\int_0^1t^{-\f{5}{6}}(1-t)^{-\f{1}{6}}dt \\
= &\frac{1}{\Gamma(\f{1}{6})\Gamma(\f{5}{6})}B(\f{1}{6}, \f{5}{6})\\
= &\f{1}{\Gamma(1)}=1.
\end{split}
\end{equation*}
Therefore,
\begin{equation}\label{equ:2.6}
  \mathbf{R}(u)(t,\rho)\geq C\int_0^t\int_{\rho-\phi(t)+\phi(s)}^{\rho+\phi(t)-\phi(s)}\big((\phi(t)+\phi(s))^2-(\rho-\rho_1)^2\big)^{-\f{1}{6}}
  \mathbf{R}(|u|^p)(s,\rho_1)\md\rho_1\md s.
\end{equation}
Notice that the support of \(u(s,\cdot)\) is contained in the ball
$B\big(0, M+\phi(s)\big)=:\{x\in \R^n: |x|\leq M+\phi(s)\}$.
On the other hand, if \(|\rho_1|>M+\phi(s)\), then for any vector \(y\in \R^n\) which is perpendicular to \(\omega\), one has
\[|\rho_1\omega+y|=\sqrt{|\rho_1|^2+y^2}\geq|\rho_1|>\phi(s)+M.\]
This yields that for \(|\rho_1|>M+\phi(s)\),
\[\mathbf{R}(|u|^p)(s,\rho_1)  = \int_{\{y: y\cdot\omega=0\}}u(s,\rho_1\omega+y)\md S_y=0.\]
Thus, \(\operatorname{supp}\mathbf{R}(|u|^p)(s,\cdot)\subseteq B\big(0, M+\phi(s)\big)\) holds.
From now on,  we can assume \(\rho\geq0\). If
\begin{equation}\label{equ:2.7}
0\leq\phi(s)\leq\phi(s_1)=:\f{\phi(t)-\rho-M}{2},
\end{equation}
then
\[\rho+\phi(t)-\phi(s)\geq\phi(s)+M, \quad \rho-\phi(t)+\phi(s)\leq-\big(\phi(s)+M\big).\]
From this, we arrive at
\begin{equation}\label{equ:2.8}
\begin{split}
\mathbf{R}(u)(t,\rho)& \geq C\int_0^{s_1}\int_{\rho-\phi(t)+\phi(s)}^{\rho+\phi(t)-\phi(s)}\big((\phi(t)+\phi(s))^2-(\rho-\rho_1)^2\big)^{-\f{1}{6}}
\mathbf{R}(|u|^p)(s,\rho_1)\md\rho_1\md s \\
& =C\int_0^{s_1}\int_{-\infty}^\infty\big((\phi(t)+\phi(s))^2-(\rho-\rho_1)^2\big)^{-\f{1}{6}}
\mathbf{R}(|u|^p)(s,\rho_1)\md\rho_1\md s.
\end{split}
\end{equation}
By \eqref{equ:2.7}, one has
\begin{align*}
\phi(t)+\phi(s)+\rho-\rho_1 & \leq\phi(t)+\phi(s)+\rho+\phi(s)+M\leq\phi(t)+\phi(t)-\rho-M+\rho+M\leq2\phi(t),\\
\phi(t)+\phi(s)-\rho+\rho_1 & \leq\phi(t)+\phi(s)-\rho+\phi(s)+M\leq2\big(\phi(t)-\rho\big).
\end{align*}
Together with this and \eqref{equ:2.8}, we deduce
\begin{multline}\label{equ:2.9}
  \mathbf{R}(u)(t,\rho)\geq C\int_0^{s_1}\int_{-\infty}^\infty\big(\phi(t)-\rho\big)^{-\f{1}{6}}\phi(t)^{-\f{1}{6}}
  \mathbf{R}(|u|^p)(s,\rho_1)\md\rho_1\md s \\
  \begin{aligned}
  & =\big(\phi(t)-\rho\big)^{-\f{1}{6}}\phi(t)^{-\f{1}{6}}\int_0^{s_1}\int_{-\infty}^\infty\int_{\{y': y'\cdot\omega=0\}}u(s,\rho_1\omega+y')\md S_{y'}\md\rho_1\md s \\
  & =\big(\phi(t)-\rho\big)^{-\f{1}{6}}\phi(t)^{-\f{1}{6}}\int_0^{s_1}\int_{\R^n}|u(s,y)|^p\md y\md s.
  \end{aligned}
\end{multline}
On the other hand, by (2.17) of \cite{HWYin1}, one has
\begin{equation}\label{equ:2.10}
  \int_{\R^n}|u(s,y)|^p\md y\geq Cs^{\frac{p}{2}}\left(M+\phi(s)\right)^{n-1-\frac{n}{2}p}.
\end{equation}
Substituting \eqref{equ:2.10} into \eqref{equ:2.9} yields
\begin{equation}\label{equ:2.11}
\begin{split}
\mathbf{R}(u)(t,\rho) & \geq C\big(\phi(t)-\rho\big)^{-\f{1}{6}}\phi(t)^{-\f{1}{6}}\int_0^{s_1}
s^{\frac{p}{2}}\left(M+\phi(s)\right)^{n-1-\frac{n}{2}p}\md s \\
& =C\big(\phi(t)-\rho\big)^{-\f{1}{6}}\phi(t)^{-\f{1}{6}}\int_0^{s_1}s^{\frac{p}{2}+\f{3}{2}\left(n-1-\frac{np}{2}\right)}\md s.
\end{split}
\end{equation}
To guarantee that the integral in \eqref{equ:2.11} is convergent, we shall need
\[\frac{p}{2}+\f{3}{2}\left(n-1-\frac{np}{2}\right)>-1.\]
This is achieved by $p=p_{\crit}(n)<p_{\conf}(n)=\frac{3n+6}{3n-2}$ and direct computation.
Thus we conclude that for \(n\geq2\),
\begin{equation}\label{equ:2.12}
\mathbf{R}(u)(t,\rho)\geq C\big(\phi(t)-\rho\big)^{-\f{1}{6}}\phi(t)^{-\f{1}{6}}\big(\phi(t)-\rho-M\big)^{n-1-\frac{np}{2}+\f{p+2}{3}}.
\end{equation}

\vskip 0.2 true cm
\textbf{Step 3. The lower bound of $\int_{\R^n}|u(t,x)|^p\md x$}
\vskip 0.2 true cm

Following (2.16) of \cite{Yor}, one can introduce the transformation
\[\mathbf{T}(f)(\rho)=\f{1}{|\phi(t)-\rho+M|^{\f{n-1}{2}}}\int_{\rho}^{\phi(t)+M}f(r)|r-\rho|^{\f{n-3}{2}}\md r\]
and derive
\begin{equation}\label{equ:2.13}
  \|\mathbf{T}(f)\|_{L^p}\leq C\|f\|_{L^p}.
\end{equation}

In fact, if \(n\geq3\), then it is easy to see that
\[\big|\mathbf{T}(f)(\rho)\big|\leq\f{2}{2|\phi(t)-\rho+M|}\int_{2\rho-\big(\phi(t)+M\big)}^{\phi(t)+M}|f(r)|\md r\leq 2M(|f|)(\rho),\]
where \(M(|f|)\) is the maximal function of \(f\). Hence there exists a constant \(C>0\) such that \eqref{equ:2.13} holds.

For \(n=2\), at first we prove that \(\mathbf{T}\) maps \(L^\infty\) to \(L^\infty\) and \(L^1\) to $L^{1,\infty}$
(weak \(L^1\) space),
respectively. If so, by the Marcinkiewicz interpolation theorem, then \eqref{equ:2.13} holds for $n=2$. 

In fact, it follows from a direct computation that for \(\rho>0\),
\begin{multline*}
  \big|\mathbf{T}(f)(\rho)\big| =\f{1}{|\phi(t)-\rho+M|^{\f{1}{2}}}\int_{\rho}^{\phi(t)+M}f(r)|r-\rho|^{-\f{1}{2}}\md r \\
  \begin{aligned}
  & \leq\f{\|f\|_{L^\infty([0,\phi(t)+M])}}{|\phi(t)-\rho+M|^{\f{1}{2}}}\int_{\rho}^{\phi(t)+M}|r-\rho|^{-\f{1}{2}}\md r \\
  & =2\|f\|_{L^\infty([0,\phi(t)+M])}\f{1}{|\phi(t)-\rho+M|^{\f{1}{2}}}|\phi(t)-\rho+M|^{\f{1}{2}} \\
  & =2\|f\|_{L^\infty([0,\phi(t)+M])},
  \end{aligned}
\end{multline*}
which yields the \(L^\infty-L^\infty\) estimate of operator $\mathbf{T}$. 
Next we derive the \(L^1-L^{1,\infty}\) estimate
of $\mathbf{T}$. Suppose
\(f\in L^1([0,\phi(t)+M])\). Let
\begin{align*}
  g(\rho) & =\f{1}{|\phi(t)-\rho+M|^{\f{1}{2}}}, \\
  h(\rho) & =\int_{\rho}^{\phi(t)+M}f(r)|r-\rho|^{-\f{1}{2}}\md r.
\end{align*}
Denote $d_\varphi(\alpha)=\big|\{0\leq \rho\leq\phi(t)+M: \varphi(\rho)>\alpha\}\big|$
as the distribution function of \(\varphi\). It is known that for \(0<\alpha<\infty\) and measurable functions \(f_1\), \(f_2\)
\[d_{f_1f_2}(\alpha)\leq d_{f_1}(\alpha^{\f{1}{2}})+d_{f_2}(\alpha^{\f{1}{2}}).\]
Note that
\[d_{g}(\alpha^{\f{1}{2}})=\big|\{0\leq \rho\leq\phi(t)+M: g(\rho)>\alpha\}\big|=\f{1}{\alpha}.\]
In addition,
\[|h(\rho)|\leq\int_{0}^{\phi(t)+M}|f(r)||r-\rho|^{-\f{1}{2}}\md r=f\ast \f{1}{|r|^{\f{1}{2}}}.\]
Since \(\f{1}{|r|^{\f{1}{2}}}\in L^{2,\infty}([0,\phi(t)+M])\) and \(f\in L^1([0,\phi(t)+M])\), by Young's inequality, we have
\(h\in L^{2,\infty}([0,\phi(t)+M])\). Therefore,
\[\alpha d_{gh}(\alpha)\leq \alpha d_{g}(\alpha^{\f{1}{2}})+\alpha d_{h}(\alpha^{\f{1}{2}})
\leq C,\]
which means \(\mathbf{T}(f)(\rho)=g(\rho)h(\rho)\in L^{1,\infty}([0,\phi(t)+M])\). 
Then an application of Marcinkiewicz interpolation theorem yields
\begin{equation}\label{equ:2.131}
\|\mathbf{T}(f)\|_{L^p([0,\phi(t)+M])}\leq C_0\|f\|_{L^p([0,\phi(t)+M])},
\end{equation}
where \(C_0>0\) is a uniform constant independent of \(t\). Due to \(\operatorname{supp}u(t,\cdot)\subseteq[0,\phi(t)+M]\), 
the inequality \eqref{equ:2.131} is enough for the application in the proof of Theorem~\ref{thm:1.1}.

Applying \eqref{equ:2.13} or \eqref{equ:2.131} to the function
\begin{equation*}
  f(r)=\left\{ \enspace
\begin{aligned}
&|u(t,r)|r^{\f{n-1}{p}}, && r\geq0,\\
&0, && r<0,
\end{aligned}
\right.
\end{equation*}
we have
\begin{align}\label{equ:2.14}
\int_{0}^{\phi(t)+M}\bigg(&\f{1}{|\phi(t)-\rho+M|^{\f{n-1}{2}}}\int_{\rho}^{\phi(t)+M}|u(t,r)|r^{\f{n-1}{p}}|r-\rho|^{\f{n-3}{2}}\md r\bigg)^p\md \rho \\
& \leq C\int_{0}^{\infty}|u(t,r)|^pr^{n-1}\md r= C\int_{\R^n}|u(t,x)|^p\md x.
\end{align}
When \(r\geq\rho\), we arrive at
\begin{equation*}
  r^{\f{n-1}{p}}=r^{\f{n-1}{2}}r^{\f{n-1}{p}-\f{n-1}{2}}\geq\left\{ \enspace
\begin{aligned}
&r^{\f{n-1}{2}}\rho^{\f{n-1}{p}-\f{n-1}{2}}, && 1<p\leq2,\\
&r^{\f{n-1}{2}}\big(\phi(t)+M\big)^{\f{n-1}{p}-\f{n-1}{2}}, && p>2.
\end{aligned}
\right.
\end{equation*}
Next we only treat  the case of \(1<p\leq2\) since the treatment for \(p>2\) is completely similar. When \(1<p\leq2\),
it follows from \eqref{equ:2.14} that
\newpage
\begin{multline}\label{equ:2.15}
  \int_{0}^{\phi(t)+M}\bigg(\f{1}{|\phi(t)-\rho+M|^{\f{n-1}{2}}}\int_{\rho}^{\phi(t)+M}|u(t,r)|r^{\f{n-1}{2}}|r-\rho|^{\f{n-3}{2}}\md r\bigg)^p\rho^{(n-1)(1-p/2)}\md\rho \\
  \leq C\int_{\R^n}|u(t,x)|^p\md x.
\end{multline}
On the other hand,
\begin{align}\label{equ:2.16}
  \mathbf{R}(u)(t,\rho)=c_n\int_{|\rho|}^{\infty}u(t,r)(r^2-\rho^2)^{\f{n-3}{2}}r\md r\leq c_n\int_{|\rho|}^{\infty}u(t,r)r^{\f{n-1}{2}}(r-\rho)^{\f{n-3}{2}}\md r.
\end{align}
Substituting \eqref{equ:2.16} into \eqref{equ:2.15} yields
\[\int_{\R^n}|u(t,x)|^p\md x\geq C\int_{0}^{\phi(t)+M}\f{\big(\mathbf{R}(u)(t,\rho)\big)^p}{\big(\phi(t)-\rho+M\big)^{\f{(n-1)p}{2}}}\rho^{(n-1)(1-p/2)}\md\rho.\]
By the bound of \(\mathbf{R}(u)\) in \eqref{equ:2.12}, we deduce
\begin{multline}\label{equ:2.17}
  \int_{\R^n}|u(t,x)|^p\md x \\
  \geq C\int_{0}^{\phi(t)+M}\f{\big(\phi(t)-\rho\big)^{-\gamma}\phi(t)^{-\gamma}\big(\phi(t)-\rho-M\big)^{p\big(n-1-\frac{np}{2}
  +\f{p+2}{3}\big)}}{\big(\phi(t)-\rho+M\big)^{\f{(n-1)p}{2}}}\rho^{(n-1)(1-p/2)}\md\rho.
\end{multline}
If \(\rho\in(0, \phi(t)-M-1)\), then there exists a constant \(C_M>0\) such that for all \(\phi(t)>2(M+1)\),
\[\phi(t)-\rho+M\leq C_M\big(\phi(t)-\rho-M\big), \quad \phi(t)-\rho\leq C_M\big(\phi(t)-\rho-M\big).\]
This observation together with \eqref{equ:2.17} yields
\begin{align}\label{A-1}
\int_{\R^n}|u(t,x)|^p\md x\geq C\int_{0}^{\phi(t)+M}\frac{\rho^{(n-1)(1-p/2)}\phi(t)^{-\f{p}{6}}}
{\big(\phi(t)-\rho-M\big)^{\big(\f{n}{2}-\f{1}{3}\big)p^2-\f{n}{2}p}}\md\rho.
\end{align}
Note that for $p=p_{crit}(n)$,
\[\bigg(\f{n}{2}-\f{1}{3}\bigg)p^2-\f{n}{2}p=1.\]
Thus we have from \eqref{A-1} that for $p=p_{crit}(n)$,
\begin{align}\label{equ:2.18}
\int_{\R^n}|u(t,x)|^p\md x &\geq C\phi(t)^{-\f{p}{6}}\int_{0}^{\phi(t)+M}\frac{\rho^{(n-1)(1-p/2)}}{\phi(t)-\rho-M}\md\rho\no \\
& \geq C\phi(t)^{-\f{p}{6}}\phi(t)^{(n-1)(1-p/2)}\int_{\f{\phi(t)-M-1}{2}}^{\phi(t)-M-1}\f{1}{\phi(t)-\rho-M}\md\rho \no\\
& \geq C\phi(t)^{n-1-\f{np}{2}+\f{p}{3}}\ln\big(\phi(t)-M+1\big).
\end{align}
Note that the term \(\ln\big(\phi(t)-M+1\big)\) can be sufficiently large when \(t\) is large, and if the power of \(t\)
in the right hand side of \eqref{equ:2.18} satisfies
\begin{equation}\label{equ:2.19}
\sigma=:\f{3}{2}\Big(n-1-\f{np}{2}+\f{p}{3}\Big)>-1,
\end{equation}
then there is a large constant \(K_0>0\) such that for large $t>0$ and $p=p_{crit}(n)$,
\[G''(t)=\int_{\R^n}|u(t,x)|^p\md x\geq K_0t^{\f{p}{2}+\f{3}{2}\big(n-1-\f{np}{2}\big)}\geq CK_0(t+M)^{\f{p}{2}+\f{3}{2}\big(n-1-\f{np}{2}\big)},\]
and
\begin{equation}\label{equ:2.20}
G(t)\geq CK_0(t+M)^{\f{p}{2}+2+\f{3}{2}\big(n-1-\f{np}{2}\big)}.
\end{equation}
Next we turn to verify \eqref{equ:2.19}. By the condition
\[p=p_{\crit}(n)<p_{\conf}(n)=\f{3n+6}{3n-2},\]
direct computation yields
\begin{equation*}
  \begin{split}
     \sigma= &\f{3}{2}(n-1)-\f{1}{2}\Big(\f{3}{2}n-1\Big)p \\
       & >\f{3}{2}(n-1)-\f{1}{2}\Big(\f{3}{2}n-1\Big)p_{\conf}(n) \\
     & =\f{3n}{4}-3.
  \end{split}
\end{equation*}
If \(n\geq3\), then
\[\sigma>-\f{3}{4}>-1.\]
If $n=2$, then
\[\sigma=\f32(1-\f23 p_{crit}(2))=\f{3-\sqrt{33}}{4}>-1.\]
Hence \eqref{equ:2.19} is valid for all $n\ge 2$. By \eqref{equ:2.20} and \eqref{1-0}, choosing $a=\f{p}{2}+\f{3}{2}\big(n-1-\f{np}{2}\big)+2$
and $q=\f{3n}{2}(p-1)$ with $p=p_{crit}(n)$ in Lemma 2.1, then  all the assumptions
of  Lemma 2.1 hold. Therefore, Theorem 1.1 is shown by Lemma 2.1.

\section{Strichartz estimates in angular mixed norm spaces}\label{sec:global}
Before establishing Strichartz estimates for the linear Tricomi
operator, we recall two important results. The first one is a minor variant of  \cite[Lemma~3.8]{Ls}, and the second one comes from
\cite[Theorem~1.2]{Chr2}.

\begin{lemma}\label{lem3.1}
Let $\ds \beta\in C_0^\infty\left((1/2,2)\right)$ and
$\ds \sum\limits_{j=-\infty}^\infty\beta\left(2^{-j}\tau\right) \equiv1$ for
$\tau>0$.  Define the Littlewood-Paley operators as
\[
G_j(t,x)=(2\pi)^{-n}\int_{\R^n}e^{ix\cdot\xi}\beta\left(2^{-j}|\xi|\right)
\hat{G}(t,\xi)\,d\xi, \quad j\in \Z.
\]
Then
\begin{align*}
\| G\|_{L^s_tL^q_x} \leq C\left(\sum\limits_{j=-\infty}^{\infty}\| G_j
\|^2_{L^s_tL^q_x}\right)^{1/2},& \quad
2\leq q<\infty,\, 1\leq s \leq \infty, \\
\intertext{and}
\left(\sum\limits_{j=-\infty}^{\infty}\| G_j\|^2_{L^r_tL^p_x}\right)^{1/2}
\leq C\| G\|_{L^r_tL^p_x},& \quad 1<p\leq2,\, 1\leq r \leq \infty.
\end{align*}
\end{lemma}
%\begin{proof}
%The proof relies on Littlewood-Paley theory. Specifically, recall that if \(1<p<\infty\) there is a constant \(C=C_p>0\) such that
%\[C^{-1}\leq\big\|\big(\sum_j|G_j(t,\cdot)|^2\big)^{\f{1}{2}}\big\|_{L^q_x}\bigg/\|G_j(t,\cdot)\|_{L^q_x}\leq C\]
%To use this, let us focus on the first inequality. If we use the first Littlewood-Paley inequality, we get that
%\[\|G(t,\cdot)\|_{L^q_x}^2\leq C\bigg(\int_{\R^2}\Big(\sum_j|G_j(t,x)|^2\Big)^{\f{q}{2}}\md x\bigg)^{\f{2}{q}}.\]
%Since \(\f{q}{2}\geq1\), Minkowski's inequality implies that the last term can be controlled by \(\sum_j\|G_j(t,\cdot)\|_{L^q_x}^2\). Now we use Minkowski's inequality again to get
%\begin{equation*}
  %\begin{split}
  %\|G\|_{L^s_tL^q_x}&=\bigg(\int_{\R}\|G_j(t,\cdot)\|_{L^q_x}^s\md t\bigg)^{\f{1}{s}} \\
  %&\qquad\qquad\leq C\bigg(\int_{\R}\Big(\sum_j\|G_j(t,\cdot)\|_{L^q_x}^2\Big)^{\f{s}{2}}\md t\bigg)^{\f{1}{s}} \\
  %&\qquad\qquad= C\bigg(\int_{\R}\Big(\sum_j\|G_j(t,\cdot)\|_{L^q_x}^2\Big)^{\f{s}{2}}\md t\bigg)^{\f{2}{s}\cdot\f{1}{2}} \\
  %&\qquad\qquad\leq C\Big(\sum_j\|G_j(t,\cdot)\|_{L^q_xL^s_t}^2\Big)^{\f{1}{2}}
  %\end{split}
%\end{equation*}
%\end{proof}

\begin{lemma}\label{lem3.2}
Suppose that $1\leq p<q\leq\infty$. Let $T: L^p(\R)\rightarrow
L^q(\R)$ be a bounded linear operator which is defined by
\[
Tf(x)=\int_{\R}K(x,y)f(y)dy,
\]
where $K(x,y)$ is locally integrable. Define
\[
\tilde{T}f(x)=\int_{-\infty}^x K(x,y)f(y)dy.
\]
Then
\[
\|\tilde{T}f\|_{L^q}\leq C_{p,q}\,\| T\|_{L^p\to L^q}\,\| f\|_{L^p}.
\]
\end{lemma}
To prove Theorem~\ref{thm:1.2}, we shall require to get certain Strichartz estimates in \(\R_+^{1+2}\) for 2-D linear
Tricomi operator. For
this purpose, we study the following linear Cauchy problem
\begin{equation}\label{equ:3.1}
\left\{ \enspace
\begin{aligned}
&\partial_t^2 u-t\triangle u=F(t,x), \quad (t,x)\in\R_+^{1+2},\\
&u(0,\cdot)=f(x),\quad \partial_tu(0,\cdot)=g(x).
\end{aligned}
\right.
\end{equation}
Note that the solution $u$ of \eqref{equ:3.1} can be written as
\[
u(t,x)=v(t,x)+w(t,x),
\]
where $v$ solves the homogeneous problem
\begin{equation}\label{equ:3.2}
\left\{ \enspace
\begin{aligned}
&\partial_t^2 v-t\triangle v=0, \quad (t,x)\in\R_+^{1+2},\\
&v(0,\cdot)=f(x),\quad \partial_tv(0,\cdot)=g(x),
\end{aligned}
\right.
\end{equation}
and $w$ solves the inhomogeneous problem with zero initial data
\begin{equation}\label{equ:3.3}
\left\{ \enspace
\begin{aligned}
&\partial_t^2 w-t\triangle w=F(t,x), \quad (t,x)\in\R_+^{1+2},\\
&w(0,\cdot)=0,\quad \partial_tw(0,\cdot)=0.
\end{aligned}
\right.
\end{equation}
Let $\dot{H}^s(\R^2)$ denote the homogeneous Sobolev space with norm
\[
\| f\|_{\dot{H}^s(\R^2)}=\left\| |D_x|^s f\right\|_{L^2(\R^2)},
\]
where
\[
|D_x|=\sqrt{-\Delta}.
\]
It follows from \cite{Yag2} that the solution $v$ of \eqref{equ:3.2} can be expressed as
\[v(t,x)=V_1(t, D_x)f(x)+V_2(t, D_x)g(x),\]
where the symbols $V_j(t, \xi)$ ($j=1,2$) of the Fourier integral operators $V_j(t, D_x)$  are
\begin{equation}\label{equ:3.6}
\begin{split}
V_1(t,|\xi|)=&\frac{\Gamma(\frac{1}{3})}{\Gamma(\frac{1}{6})}\biggl[e^{\frac{z}{2}}H_+\Big(\frac{1}{6},\frac{1}{3};z\Big) +e^{-\frac{z}{2}}H_-\Big(\frac{1}{6},\frac{1}{3};z\Big)\biggr]
\end{split}
\end{equation}
and
\begin{equation}
\begin{split}
V_2(t,|\xi|)=&\frac{\Gamma(\frac{5}{3})}{\Gamma(\frac{5}{6})}t\biggl[
e^{\frac{z}{2}}H_+\Big(\frac{5}{6},\frac{5}{3};z\Big)
+e^{-\frac{z}{2}}H_-\Big(\frac{5}{6},\frac{5}{3};z\Big)\biggr],
\end{split}
\label{equ:3.7}
\end{equation}
here $z=2i\phi(t)|\xi|$, $i=\sqrt{-1}$, and $H_{\pm}$ are smooth functions of the variable $z$.
By \cite{Tani}, one knows that for $\beta\in\mathbb{N}_0^n$,
\begin{align}
\big| \partial_\xi^\beta H_{+}(\alpha,\gamma;z)
\big|&\leq C(\phi(t)|\xi|)^{\alpha-\gamma}(1+|\xi|^2)^{-\frac{|\beta|}{2}}
\quad if \quad \phi(t)|\xi|\geq 1, \label{equ:3.8} \\
\big| \partial_\xi^\beta H_{-}(\alpha,\gamma;z)\big|&\leq C(\phi(t)|\xi|)^{-\alpha}(1+|\xi|^2)^{-\frac{|\beta|}{2}}
\quad if \quad \phi(t)|\xi|\geq 1. \label{equ:3.9}
\end{align}
We only estimate $V_1(t, D_x)f(x)$ since the estimation on $V_2(t, D_x)g(x)$ is similar. Indeed, up
to a factor of $t\,\phi(t)^{-\frac{5}{6}}=
C\phi(t)^{-\frac{1}{6}}$, the powers of $t$ appearing in $V_1(t, D_x)f(x)$
or $V_2(t, D_x)g(x)$ are the same.

Choose a cut-off function $\chi(s)\in C^{\infty}(\Bbb R)$
with $\chi(s)=
\left\{ \enspace
\begin{aligned}
1, \quad &s\geq2 \\
0, \quad &s\leq1
\end{aligned}
\right.$. Then
\begin{equation}
\begin{split}
V_1(t,|\xi|)\hat{f}(\xi)&=\chi(\phi(t)|\xi|)V_1(t,|\xi|)\hat{f}(\xi)+(1-\chi(\phi(t)|\xi|))V_1(t,|\xi|)\hat{f}(\xi) \\
&=:\hat{v}_1(t,\xi)+\hat{v}_2(t,\xi).
\end{split}
\label{equ:3.10}
\end{equation}
By \eqref{equ:3.6}, \eqref{equ:3.8} and \eqref{equ:3.9}, we derive that
\begin{equation}
{v}_1(t,x)=C\biggl(\int_{\mathbb{R}^n}e^{i(x\cdot\xi+\phi(t)|\xi|)}a_{11}(t,\xi)\hat{f}(\xi)\md\xi+
\int_{\mathbb{R}^n}e^{i(x\cdot\xi-\phi(t)|\xi|)}a_{12}(t,\xi)\hat{f}(\xi)\md\xi\biggr), \label{equ:3.11}
\end{equation}
where $C>0$ is a generic constant, and for $\beta\in\mathbb{N}_0^n$,
\begin{equation*}
\big| \partial_\xi^\beta a_{1l}(t,\xi)\big|\leq C_{l\beta}|\xi|^{-|\beta|}\big(1+\phi(t)|\xi|\big)^{-\frac{1}{6}},
\qquad l=1,2.
\end{equation*}
Next we analyze $v_2(t,x)$. It follows from \cite{Erd1} or \cite{Yag2} that
\begin{equation*}
V_1(t,|\xi|)=e^{-\frac{z}{2}}\Phi\Big(\frac{1}{6},\frac{1}{3};z\Big), %\label{equ:2.9}
\end{equation*}
where $\Phi$ is the confluent hypergeometric function which is analytic with respect to the variable
$z=2i\phi(t)|\xi|$. Then
\begin{equation*}
\Big|\partial_\xi\big\{\big(1-\chi(\phi(t)|\xi|)\big)V_1(t,|\xi|)\big\}\Big|\leq C(1+\phi(t)|\xi|)^{-\frac{1}{6}}|\xi|^{-1}.
\end{equation*}
Similarly, one has
\begin{equation*}
\Big|\partial_\xi^{\beta}\big\{\big(1-\chi(\phi(t)|\xi|)\big)V_1(t,|\xi|)\big\}\Big|\leq
C(1+\phi(t)|\xi|)^{-\frac{1}{6}}|\xi|^{-|\beta|}.
\end{equation*}
Thus we arrive at
\begin{equation}
v_2(t,x)=C\biggl(\int_{\mathbb{R}^n}e^{i(x\cdot\xi+\phi(t)|\xi|)}a_{21}(t,\xi)\hat{f}(\xi)\md\xi
+\int_{\mathbb{R}^n}e^{i(x\cdot\xi-\phi(t)|\xi|)}a_{22}(t,\xi)\hat{f}(\xi)\md\xi\biggr), \label{equ:3.12}
\end{equation}
where, for $\beta\in\mathbb{N}_0^n$,
\begin{equation*}
\big| \partial_\xi^\beta a_{2l}(t,\xi)\big|\leq C_{l\beta}\big(1+\phi(t)|\xi|\big)^{-\frac{1}{6}}|\xi|^{-|\beta|},
\qquad l=1,2.
\end{equation*}
Substituting \eqref{equ:3.11} and \eqref{equ:3.12} into \eqref{equ:3.10} yields
\[V_1(t, D_x)f(x)=C\biggl(\int_{\mathbb{R}^n}e^{i(x\cdot\xi+\phi(t)|\xi|)}a_1(t,\xi)\hat{f}(\xi)\md\xi
+\int_{\mathbb{R}^n}e^{i(x\cdot\xi-\phi(t)|\xi|)}a_2(t,\xi)\hat{f}(\xi)\md\xi\biggr),\]
where $a_l$ $(l=1,2)$ satisfies
\begin{equation}
|\partial_\xi^\beta a_l(t,\xi)\big|\leq C_{l\beta}\big(1+\phi(t)|\xi|\big)^{-\frac{1}{6}}|\xi|^{-|\beta|}. \label{equ:3.13}
\end{equation}
Next we only treat the integral
$\int_{\R^n}e^{i\left(x\cdot\xi-\phi(t)|\xi|\right)}
a_2(t,\xi)\hat{f}(\xi)\,d\xi$ since the treatment of the
integral \linebreak
$\int_{\R^n}e^{i\left(x\cdot\xi+\phi(t)|\xi|\right)} a_1(t,\xi)
\hat{f}(\xi)\,d\xi$ is similar. Denote
\begin{equation}\label{equ:3.14}
(Af)(t,x)=\int_{\R^n}e^{i\left(x\cdot\xi-\phi(t)|\xi|\right)}a_2(t,\xi)
  \hat{f}(\xi)\,d\xi.
\end{equation}
We will show that
\begin{equation}\label{equ:3.15}
  \|(Af)(t,x)\|_{L^q_tL^r_{|x|}L_\theta^2(\R_+^{1+2})}\le C\,\| f\|_{\dot{H}^s(\R^2)},
\end{equation}
%If $g \equiv 0$ in \eqref{equ:3.2}, we intend to establish
%the Strichartz-type inequality
%\[
%\| v\|_{L^q_tL^r_{|x|}L_\theta^2(\mathbb{R}_+\times\R^2)}\leq C\,\| f\|_{\dot{H}^s(\R^n)},
%\]
where $q\ge 2$ and $r\ge 2$ are some suitable constants related to $s$.
One obtains by a scaling argument that those indices in \eqref{equ:3.15}  should satisfy
\begin{equation}\label{equ:3.4}
  \frac{1}{q}+\frac{3}{r}
  =\frac{3}{2}\left(1-s\right).
\end{equation}
On the other hand, by another scaling argument similar to Knapp's counter example, we get the second restriction on the indices
in \eqref{equ:3.15}
\begin{equation}\label{equ:3.5}
  \f{1}{q}\leq 1-\frac{3}{2}\cdot\frac{1}{r}.
\end{equation}
In fact, for small \(\delta>0\), set \(\xi=(\xi_1, \xi_2)\in\R^2\backslash\{0\}\) and denote
\begin{equation*}%\label{equ:3.18}
D=D_\delta=\{\xi\in\R^2: |\xi_1-1|<1/2, |\xi_2|<\delta\}.
\end{equation*}
Let \(\hat{f}(\xi)=\chi_D(\xi)\) be the characteristic function of domain \(D\). Note that on domain \(D\) it holds
\[|\xi|-\xi_1=\frac{|\xi|^2-\xi_1^2}{|\xi|+\xi_1}=\frac{|\xi_2|^2}{|\xi|+\xi_1}\sim\delta^2.\]
By \eqref{equ:3.14}, one has
\begin{equation}\label{equ:3.16}
\begin{split}
(Af)(t,x)&=\int_{\R^2}e^{i\left(x\cdot\xi-\phi(t)|\xi|\right)}a_2(t,\xi)\hat{f}(\xi)\ \md\xi \\
&=e^{i(x_1-\phi(t))}\int_{D}e^{i\left(-\phi(t)(|\xi|-\xi_1)+(x_1-\phi(t))(\xi_1-1)+x_2\xi_2\right)}a_2(t,\xi)\ \md\xi.
\end{split}
\end{equation}
Choose a domain $R$ in \(\mathbb{R}_+\times\R^2\) as
\begin{equation*}%\label{3.17}
  R=\{(t,x): \phi(t)\leq\delta^{-1}, |x_1-\phi(t)|\lesssim1, |x_2|\lesssim\delta^{-1}\}.
\end{equation*}
For \((t,x)\in R\) and \(\xi\in D\), then the phase function in \eqref{equ:3.16} is essentially equivalent to
a constant and we have
\[|(Af)(t,x)|\geq|D|(1+\delta^{-1})^{-\f{1}{6}}\sim|D|\delta^{\f{1}{6}}.\]
Therefore if we take \(s=0\) in \eqref{equ:3.15}, then a direct computation yields
\[\f{\|(Af)(t,x)\|_{L^q_tL^r_{|x|}L_\theta^2(\mathbb{R}_+\times\R^2)}}{\|f\|_{L^2(\R^2)}}
\geq\f{|D|\delta^{\f{1}{6}}\|\chi_R\|_{L^q_tL^r_{|x|}L_\theta^2(\mathbb{R}_+\times\R^2)}}{|D|^{\f{1}{2}}}
\sim\delta^{\frac{2}{3}-\frac{2}{3q}-\frac{1}{r}}.\]
Since \(\delta>0\) is small, in order to get \eqref{equ:3.15}, we shall need
\[\frac{2}{3}-\frac{2}{3q}-\frac{1}{r}\geq0\Longleftrightarrow\f{1}{q}\leq 1-\frac{3}{2}\cdot\frac{1}{r},\]
which gives restriction \eqref{equ:3.5}.
Now our task is to prove
\begin{lemma}\label{pro:3.1}
Let operator \(A\) be defined by \eqref{equ:3.14}. Assume that \((q,r)\neq(\infty,\infty)\),
\begin{equation*}
q, r\geq2\quad and \quad \frac{1}{q}\leq 1-\frac{3}{2}\cdot\frac{1}{r}.
\end{equation*}
Then
\begin{equation}\label{equ:3.17}
  \|(Af)(t,x)\|_{L^q_tL^r_{|x|}L_\theta^2(\R_+^{1+2})}\le C\,\| f\|_{\dot{H}^s(\R^2)},
\end{equation}
where $s=2(\f{1}{2}-\f{1}{r})-\f{2}{3}\cdot\f{1}{q}.$
\end{lemma}

\begin{proof}
The main step in the proof of \eqref{equ:3.17} is to show that
\begin{equation}\label{equ:3.18}
\begin{split}
  \|(Af)(t,x)\|&_{L^q_tL^r_{|x|}L_\theta^2(\R_+^{1+2})}\le C\,\| f\|_{L^2(\R^2)} \\
  &\text{if} \quad 2\leq q<\infty, 2\leq r\leq\infty \quad\text{and} \quad \hat{f}(\xi)=0 \quad \text{if} \quad |\xi|\notin[\f{1}{2}, 1].
\end{split}
\end{equation}
Indeed, once \eqref{equ:3.18} is proved, then by the support condition of \(f\), we know that
\begin{equation}\label{equ:3.19}
\|(Af)(t,x)\|_{L^q_tL^r_{|x|}L_\theta^2(\R_+^{1+2})}\le C\,\| f\|_{\dot{H}^s(\R^2)}, \quad
s=2(\f{1}{2}-\f{1}{r})-\f{2}{3}\cdot\f{1}{q}.
\end{equation}
This together with Lemma~\ref{lem3.1} yields \eqref{equ:3.17}.

To prove \eqref{equ:3.18}, we follow some ideas of \cite{Smi} and use the interpolation method.
The first case is \(q=\infty\) and \(s=1-\f{2}{r}\). Since Hardy-Littlewood-Sobolev estimate gives \(\dot{H}^{1-\f{2}{r}}(\R^2)\subseteq L^r(\R^2)\)
for \(2\leq r<\infty\), we clearly have
\begin{equation*}%\label{equ:3.20}
   \|Af\|_{L^\infty_tL^r_{|x|}L_\theta^2(\R_+^{1+2})}\le C\,\| Af\|_{L^\infty_tL^r_{|x|}L_\theta^r(\R_+^{1+2})}
   \leq C\ \|Af\|_{L^\infty_t\dot{H}^{1-\f{2}{r}}(\R_+^{1+2})}.
\end{equation*}
It follows from \eqref{equ:3.13} and \eqref{equ:3.14} that if $\hat{f}(\xi)=0$ for $|\xi|\notin[\f{1}{2}, 1]$,
\begin{align*}
  &\|Af\|_{L^\infty_t\dot{H}^{1-\f{2}{r}}(\mathbb{R}_+\times\R^2)} \\
  &\leq C\ \bigg\|\Big\|\int_{\R^2}e^{i\left(x\cdot\xi+\phi(t)|\xi|\right)}a_1(t,\xi)
  \hat{f}(\xi)\,d\xi\Big\|_{\dot{H}^{1-\f{2}{r}}(
  \R^2)}\bigg\|_{L^\infty_t(\R)} \\
  &\leq C\ \bigg\|\Big\||\xi|^{1-\f{2}{r}}\big(1+\phi(t)|\xi|\big)^{-\f{1}{6}}|\hat{f}(\xi)|\Big\|_{L^2(
  \R^2)}\bigg\|_{L^\infty_t(\R)} \\
  &\leq C\ \|f\|_{\dot{H}^{1-\f{2}{r}}(\R^2)}\leq C\ \|f\|_{L^2(\R^2)}.
\end{align*}
By interpolation, if we can conclude that for \(2\leq q<\infty\),
\begin{equation}\label{equ:3.20}
\|Af\|_{L^q_tL^\infty_{|x|}L_\theta^2(\mathbb{R}_+\times\R^2)}\le C\,\| f\|_{L^2(\R^2)} \quad
\text{for $\hat{f}(\xi)=0$ for $|\xi|\notin[\f{1}{2}, 1]$},
\end{equation}
then \eqref{equ:3.18} is immediately proved.
Next we turn to the proof of \eqref{equ:3.20}. By the support condition for \(\hat{f}\), we arrive at
\begin{equation}\label{equ:3.21}
  \|f\|_{L^2(\R^2)}^2\approx\int_{0}^{\infty}\int_{0}^{2\pi}|\hat{f}(\rho\cos\omega, \rho\sin\omega)|^2\ \md\omega\md\rho,
\end{equation}
here \(\xi=(\rho\cos\omega, \rho\sin\omega)\).
Expanding the angular part of \(\hat{f}\) by Fourier series yields that there are coefficients \(c_k(\rho)\) 
$(k\in\Bbb Z)$
vanishing for \(\rho\notin[\f{1}{2}, 1]\) such that
\[\hat{f}(\xi)=\sum_kc_k(\rho)e^{ik\omega}.\]
This means
\begin{equation}\label{equ:3.24}
  (Af)(t,\xi)=e^{-i\phi(t)\rho}a_2(t,\rho)\sum_kc_k(\rho)e^{ik\omega}.
\end{equation}
By Plancherel's theorem for \(\mathbb{S}^1\) and \(\mathbb{R}\), we have
\begin{equation}\label{equ:3.22}
  \|f\|_{L^2(\R^2)}^2\approx\sum_k\int_{\R}|c_k(\rho)|^2\ \md\rho\approx\sum_k\int_{\R}|\hat{c}_k(s)|^2\ \md s,
\end{equation}
where \(\hat{c}_k(s)\) is the one-dimensional Fourier transform of \(c_k(\rho)\). Recall that (see \cite{Ste}, p.137)
\begin{equation}\label{equ:3.23}
  f\big(r(\cos\omega, \sin\omega)\big)=(2\pi)^{-1}\sum_k\Big(i^k\int_{0}^{\infty}J_k(r\rho)c_k(\rho)\rho\ \md \rho\Big)e^{ik\omega},
\end{equation}
where  \(k\in\mathbb{Z}\), and \(J_k\) is the \(k\)-th Bessel function defined by
\[J_k(y)=\frac{(-i)^k}{2\pi}\int_{0}^{2\pi}e^{iy\cos\theta-ik\theta}\ \md\theta.\]
Choose a cut-off function \(\beta\in C_0^\infty(\mathbb{R})\) such that
\begin{equation*}
  \beta(\tau)=
  \left\{ \enspace
  \begin{aligned}
  &1, && \frac{1}{2}\leq\tau\leq1,\\
  &0, && \tau\notin[\f{1}{4}, 2].
  \end{aligned}
  \right.
\end{equation*}
Let \(\alpha(t,\rho)=\rho\beta(\rho)a_2(t,\rho)\). Then by \eqref{equ:3.23} and the support condition of \(c_k\), we have
\begin{multline*}
  (Af)\big(t, r(\cos\omega, \sin\omega)\big) \\
  \begin{aligned}
  &=(2\pi)^{-1}\sum_k\bigg(i^k\int_{0}^{\infty}J_k(r\rho)e^{-i\phi(t)\rho}c_k(\rho)\beta(\rho)a_2(t,\rho)\rho\ \md\rho\bigg)e^{ik\omega} \\
  &=(2\pi)^{-2}\sum_k\bigg(i^k\int_{0}^{\infty}\int_{-\infty}^{\infty}J_k(r\rho)e^{i\rho\big(s-\phi(t)\big)}\hat{c}_k(s)\alpha(t,\rho)\ \md s\ \md\rho\bigg)e^{ik\omega} \\
  &=(2\pi)^{-3}\sum_k\bigg(i^k\int_{0}^{\infty}\int_{-\infty}^{\infty}\int_{0}^{2\pi}e^{i\rho r\cos\theta-ik\theta}e^{i\rho\big(s-\phi(t)\big)}\hat{c}_k(s)\alpha(t,\rho)\ \md\theta\ \md s\ \md\rho\bigg)e^{ik\omega} \\
  &=(2\pi)^{-3}\sum_k\bigg(i^k\int_{-\infty}^{\infty}\int_{0}^{2\pi}e^{-ik\theta}\hat{\alpha}_t\big(\phi(t)-s-r\cos\theta\big)
  \hat{c}_k(s)\ \md\theta\ \md s\bigg)e^{ik\omega},\\
  \end{aligned}
\end{multline*}
where $\hat{\alpha}_t(\xi)$ stands for the Fourier transformation of $\al(t,\rho)$ with respect to the variable $\rho$.
Direct computation yields that for any \(r\geq0\),
\begin{align}\label{equ:3.25}
&\int_{0}^{2\pi}\big|(Af)\big(t, r(\cos\omega, \sin\omega)\big)\big|^2\md \omega \no\\
&=(2\pi)^{-5}\sum_k\Big|\int_{-\infty}^{\infty}\int_{0}^{2\pi}e^{-ik\theta}\hat{\alpha}_t\big(\phi(t)-s-r\cos\theta\big)
\hat{c}_k(s)\ \md\theta\ \md s\Big|^2.
\end{align}
To proceed further, we shall need a control of the integral in \eqref{equ:3.25} with respect to the variable \(\theta\),
which is similar to Lemma 2.1 in \cite{Smi}.
\begin{lemma}\label{lem:3.3}
  Let \(\hat{\alpha}(t,\xi)\) be defined as above and a number \(N\in\N\) be fixed. Then there is a uniform constant \(C>0\), which is 
  independent of the variables \(b\in\R\) and \(r\geq0\), so that the following inequalities hold:
  \begin{equation}\label{equ:3.26}
    \int_{0}^{2\pi}|\hat{\alpha}_t(b-r\cos\theta)|\ \md \theta\leq C\langle b\rangle^{-N}\big(1+\phi(t)\big)^{-\frac{1}{6}} \qquad \text{if} \quad 0\leq r\leq1 \quad \text{or} \quad |b|\geq2r;
  \end{equation}
    \begin{equation}\label{equ:3.27}
    \int_{0}^{2\pi}|\hat{\alpha}_t(b-r\cos\theta)|\ \md \theta\leq C\big(r^{-1}+r^{-\f{1}{2}}\langle r-|b|\rangle^{-\f{1}{2}}\big)\big(1+\phi(t)\big)^{-\frac{1}{6}}
    \qquad \text{if $r>1$ and $|b|\leq2r$}.
   \end{equation}
\end{lemma}
\begin{proof}[Proof of Lemma 3.4.]
By the definition of function $\hat{\alpha}_t$, we only need to study the integral
\[I=\int_{0}^{2\pi}\Big|\int_{0}^{\infty}e^{-i(b-r\cos\theta)}\rho\beta(\rho)a_2(t,\rho)\md \rho\Big|\ \md \theta.\]
\vskip 0.2 true cm

\textbf{Case I.} \(\mathbf{|b|\geq2r}\)
\vskip 0.2 true cm
In this case, we have
\begin{align*}
I=&\int_{0}^{2\pi}\Big|\int_{0}^{\infty}(b-r\cos\theta)^N e^{-i(b-r\cos\theta)}\rho\beta(\rho)a_2(t,\rho)\md \rho\Big|(b-r\cos\theta)^{-N}\ \md \theta \\
&=\int_{0}^{2\pi}\Big|\int_{0}^{\infty}(-D_\rho)^N\big(e^{-i(b-r\cos\theta)}\big)\rho\beta(\rho)a_2(t,\rho)\md \rho\Big|(b-r\cos\theta)^{-N}\ \md \theta  \\
&=\int_{0}^{2\pi}\Big|\int_{0}^{\infty}e^{-i(b-r\cos\theta)}(D_\rho)^N\big(\rho\beta(\rho)a_2(t,\rho)\big)\md \rho\Big|(b-r\cos\theta)^{-N}\ \md \theta.
\end{align*}
Since \(\rho\beta(\rho)\in\mathcal{S}(\R)\) and \(a_2\) satisfies \eqref{equ:3.13}, direct computation yields
\[I\leq\int_{0}^{2\pi}\big(1+\phi(t)\big)^{-\frac{1}{6}}\langle b\rangle^{-N}\ \md\theta
\leq C\big(1+\phi(t)\big)^{-\frac{1}{6}}\langle b\rangle^{-N},\]
which just corresponds to \eqref{equ:3.26}.

\vskip 0.2 true cm

\textbf{Case II.} \(\mathbf{0\leq r\leq1}\)
\vskip 0.2 true cm

For \(|b|>2\), it is reduced to Case I. For \(|b|\leq2\), by a direct computation, we have
\begin{align*}
|I|&\leq\int_{0}^{2\pi}\Big|\int_{1/4}^{2}\rho\beta(\rho)\big(1+\phi(t)\rho\big)^{-\frac{1}{6}}\ \md\rho\Big|\ \md\theta \\
&\leq C\big(1+\phi(t)\big)^{-\frac{1}{6}}\leq C\big(1+\phi(t)\big)^{-\frac{1}{6}}\langle b\rangle^{-N}.
\end{align*}

\vskip 0.2 true cm
\textbf{Case III.} \(\mathbf{r>1}\) \textbf{and} \(\mathbf{|b|\leq 2r}\)
\vskip 0.2 true cm
In this case, we intend to prove that
\begin{equation}\label{equ:3.28}
  \int_{0}^{\pi/4}|\hat{\alpha}_t(b-r\cos\theta)|\ \md \theta+\int_{3\pi/4}^{\pi}|\hat{\alpha}_t(b-r\cos\theta)|\ \md \theta\leq Cr^{-\f{1}{2}}\langle r-|b|\rangle^{-\f{1}{2}}\big(1+\phi(t)\big)^{-\frac{1}{6}},
\end{equation}
and
\begin{equation}\label{equ:3.29}
  \int_{\pi/4}^{3\pi/4}|\hat{\alpha}_t(b-r\cos\theta)|\ \md \theta\leq r^{-1}\big(1+\phi(t)\big)^{-\frac{1}{6}}.
\end{equation}

To show \eqref{equ:3.28}, it only suffices to estimate the first integral in \eqref{equ:3.28}. Let \(u=1-\cos\theta\), we then have
\begin{align}\label{B-0}
&\int_{0}^{\pi/4}|\hat{\alpha}_t(b-r\cos\theta)|\ \md \theta=\int_{0}^{1-\f{\sqrt{2}}{2}}|\hat{\alpha}_t(b-r+ru)|\ \f{\md u}{\sqrt{2u-u^2}}\no\\
&\leq C\int_{0}^{1-\f{\sqrt{2}}{2}}|\hat{\alpha}_t(b-r+ru)|\ \f{\md u}{\sqrt{u}}.
\end{align}
We further set \(\bar{u}=ru\), then the last integral in \eqref{B-0} can be controlled by
\begin{align}\label{B-0'}
&r^{-\f{1}{2}}\int_{0}^\infty|\hat{\alpha}_t(b-r+\bar{u})|\ \f{\md \bar u}{\sqrt{\bar u}}\no\\
&=r^{-\f{1}{2}}\int_{0}^\infty\Big|\int_{-\infty}^{\infty}e^{-i(b-r+\bar{u})}\rho\beta(\rho)a_2(t,\rho)\md \rho\Big|\ \f{\md \bar{u}}{\sqrt{\bar{u}}}
=:r^{-\f{1}{2}}II.
\end{align}
If \(\big|r-|b|\big|\geq2\), then
\begin{align}\label{B-1}
II=&\int_{0}^{\big|r-|b|\big|/2}\Big|\int_{-\infty}^{\infty}e^{-i(b-r+\bar{u})}\rho\beta(\rho)a_2(t,\rho)\md \rho\Big|\ \f{\md \bar{u}}{\sqrt{\bar{u}}}\no\\
&+\int_{\big|r-|b|\big|/2}^\infty\Big|\int_{-\infty}^{\infty}e^{-i(b-r+\bar{u})}\rho\beta(\rho)a_2(t,\rho)\md \rho\Big|\ \f{\md \bar{u}}{\sqrt{\bar{u}}}\no\\
&=:II_1+II_2.
\end{align}
For \(II_1\), we can repeat the analysis in Case I and integrate by parts to get
\begin{align}\label{B-2}
  II_1 & \leq C\int_{0}^{\big|r-|b|\big|/2}|b-r+\bar{u}|^{-N}\big(1+\phi(t)\big)^{-\frac{1}{6}}\f{\md \bar{u}}{\sqrt{\bar{u}}}\no \\
   & \qquad \leq C\big||b|-r\big|^{-N}\big(1+\phi(t)\big)^{-\frac{1}{6}}\int_{0}^{\big|r-|b|\big|/2}\f{\md \bar{u}}{\sqrt{\bar{u}}} \no \\
   & \qquad \leq C\langle |b|-r\rangle^{\f{1}{2}-N}\big(1+\phi(t)\big)^{-\frac{1}{6}}.
\end{align}
For \(II_2\), integrating by parts yields
\begin{align}\label{B-3}
  II_2 & \leq C\big(1+\phi(t)\big)^{-\frac{1}{6}}\ \int_{\big||b|-r\big|/2}^{\infty}\langle |b|-r+\bar{u}\rangle^{-N}\f{\md \bar{u}}{\sqrt{\bar{u}}}\no\\
    & \leq C\big(1+\phi(t)\big)^{-\frac{1}{6}}\ \big|r-|b|\big|^{-\f{1}{2}}\int_{\big||b|-r\big|/2}^{\infty}\langle |b|-r+\bar{u}\rangle^{-N}\ \md \bar{u} \no\\
    & \leq C\big(1+\phi(t)\big)^{-\frac{1}{6}}\ \big||b|-r\big|^{-\f{1}{2}}.
\end{align}
If \(\big|r-|b|\big|\leq2\), then by similar computation,
\begin{align}\label{B-4}
  &\int_{0}^{\infty}\big|\hat{\alpha}_t(b-r+\bar{u})\big|\f{\md \bar{u}}{\sqrt{\bar{u}}}
  \leq C\ \int_{0}^{\infty}\langle b-r+\bar{u}\rangle^{-N}\big(1+\phi(t)\big)^{-\frac{1}{6}}\f{\md \bar{u}}{\sqrt{\bar{u}}}\no\\
  &\leq C\ \big(1+\phi(t)\big)^{-\frac{1}{6}}\langle r-|b|\rangle^{-\frac{1}{2}}.
\end{align}
Thus it follows from \eqref{B-0}-\eqref{B-4} that the proof of \eqref{equ:3.28} is finished.

To show \eqref{equ:3.29}, we set \(u=r\cos\theta\). Then
\begin{align*}
  &\int_{\pi/4}^{3\pi/4}|\hat{\alpha}_t(b-r\cos\theta)|\ \md \theta=\int_{-\f{\sqrt{2}}{2}r}^{\f{\sqrt{2}}{2}r}\Big|\int_{-\infty}^{\infty}e^{-i\rho(b-u)}\alpha(t,\rho)\md \rho\Big|\f{\md u}{r\sin\theta} \\
  &\leq C\ r^{-1}\int_{-\f{\sqrt{2}}{2}r}^{\f{\sqrt{2}}{2}r}\big(1+\phi(t)\big)^{-\frac{1}{6}}\langle b-u\rangle^{-N}\md u \\
  &\leq C\ r^{-1}\int_{-\infty}^{\infty}\big(1+\phi(t)\big)^{-\frac{1}{6}}\langle u\rangle^{-N}\md u \\
  &\leq C\ \big(1+\phi(t)\big)^{-\frac{1}{6}}r^{-1}.
\end{align*}
Collecting all the analysis above in Case I-Case III yields the proof of Lemma~\ref{lem:3.3}.
\end{proof}

By Lemma~\ref{lem:3.3}, we have
\begin{claim}
For \(\delta>0\), there is a constant \(C_\delta>0\), which is independent of \(t\in \R_+\) and \(r\geq0\), such that
\begin{equation}\label{equ:3.30}
  \int_{-\infty}^{\infty}\bigg(\int_{0}^{2\pi}\big(1+\phi(t)\big)^{\frac{1}{6}}\langle\phi(t)-s\rangle^{\f{1}{2}-\delta}
  \big|\hat{\alpha}_t\big(\phi(t)-s-r\cos\theta\big)\big|\md\theta\bigg)^2\md s\leq C_\delta.
\end{equation}
\end{claim}

\begin{proof}[proof of claim]
If \(0\leq r\leq1\), or \(|\phi(t)-s|\geq2r\), then for \(\delta>0\), it is easy to see that \eqref{equ:3.26} yields the expected estimate
\eqref{equ:3.30}.

If \(|\phi(t)-s|\leq2r\) and \(r>1\), then by \eqref{equ:3.27}, one has
\begin{align}\label{B-5}
  &\int_{\phi(t)-2r}^{\phi(t)+r}\bigg(\int_{0}^{2\pi}\big(1+\phi(t)\big)^{-\frac{1}{6}}\langle\phi(t)-s\rangle^{\f{1}{2}-\delta}
  \big|\hat{\alpha}_t\big(\phi(t)-s-r\cos\theta\big)\big|\md\theta\bigg)^2\md s\no \\
  &\leq\int_{\phi(t)-2r}^{\phi(t)+r}\bigg(\big(1+\phi(t)\big)^{\frac{1}{6}}\langle\phi(t)-s\rangle^{\f{1}{2}-\delta}
  (\big(1+\phi(t)\big)^{-\frac{1}{6}}
  \big(r^{-1}+r^{-\f{1}{2}}\langle r-|\phi(t)-s|\rangle^{-\f{1}{2}}\big)\bigg)^2\md s\no \\
  &\leq 2\int_{\phi(t)-2r}^{\phi(t)+r}\big(r^{-1-2\delta}+\langle\phi(t)-s\rangle^{1-2\delta}r^{-1}\big\langle r-|\phi(t)-s|\big\rangle^{-1}\big)\md s.
  \end{align}
Next we treat the integral in \eqref{B-5}. By \(r>1\), we have
\[\int_{\phi(t)-2r}^{\phi(t)+2r}r^{-1-2\delta}ds=\frac{4}{r^{2\delta}}\leq C_\delta.\]
For the second part of the integral in \eqref{B-5}, if \(|\phi(t)-s|\leq1\) and  \(r>1\), then
\[\int_{\phi(t)-2r}^{\phi(t)+2r}\langle\phi(t)-s\rangle^{1-2\delta}r^{-1}\big\langle r-|\phi(t)-s|\big\rangle^{-1}\md s
\leq C\int_{\phi(t)-1}^{\phi(t)+1}r^{-1}\md s\leq \f{C}{r}\leq C_\delta.\]
If \(|\phi(t)-s|>1\) and  \(r>1\), we then set \(\eta=|\phi(t)-s|\) and derive that for \(0<\delta<1/2\), the integral in \eqref{B-5} can be controlled by
\begin{align}\label{equ:3.31}
&\bigg|\int_{\phi(t)-2r}^{\phi(t)+r}\langle\phi(t)-s\rangle^{1-2\delta}r^{-1}\big\langle r-|\phi(t)-s|\big\rangle^{-1}\md s\bigg|\no\\
&\leq C\int_{1}^{2r}r^{-1}\eta^{1-2\delta}\langle r-\eta\rangle^{-1}\md \eta\no\\
&\leq \f{C}{r}\int_{1}^{2r}\f{\eta^{1-2\delta}}{1+|r-\eta|}\md \eta \no\\
&\leq \f{C}{r}\Big(\int_{1}^{r}\f{\eta^{1-2\delta}}{1+r-\eta}\md \eta
+\int_{r}^{2r}\f{\eta^{1-2\delta}}{1+\eta-r}\md \eta\Big).
\end{align}
Note that
\begin{equation}\label{equ:3.32}
  \begin{split}
     \int_{1}^{r}\f{\eta^{1-2\delta}}{1+r-\eta}\md \eta & =\ln(2+r)+(1-2\delta)\int_{1}^{r}\ln(1+r-\eta)\eta^{-2\delta}\md \eta \\
       & \leq \ln(2+r)+(1-2\delta)\ln r\int_{1}^{r}\eta^{-2\delta}\md \eta \\
       & \leq \ln(2+r)+r^{1-2\delta}\ln r\\
  \end{split}
\end{equation}
and
\begin{align}\label{equ:3.33}
&\int_{r}^{2r}\f{\eta^{1-2\delta}}{1+\eta-r}\md \eta=\eta^{1-2\delta}\ln(1+\eta-r)\big|_r^{2r}-\int_{r}^{2r}\ln(1+\eta-r)\md \eta^{1-2\delta}\no\\
&\leq (2r)^{1-2\delta}\ln(1+r).
\end{align}
Then collecting \eqref{equ:3.31}, \eqref{equ:3.32} and \eqref{equ:3.33} yields
\begin{align*}
&\int_{1}^{2r}r^{-1}\eta^{1-2\delta}\langle r-\eta\rangle^{-1}\md \eta\leq \f{\ln(1+r)}{r^{2\delta}}\leq C_\delta.
\end{align*}
Hence the claim is proved.
\end{proof}

It follows from {\bf Claim}, \eqref{equ:3.25} and H\"{o}lder's inequality that
\[\|Af\|_{L_{|x|}^\infty L_\theta^2}^2\leq C_\delta\sum_k\int_{-\infty}^{\infty}\big|\big(1+\phi(t)\big)^{-\frac{1}{6}}\langle\phi(t)-s\rangle^{-\f{1}{2}+\delta}
\hat{c}_k(s)\big|^2\md s.\]
For any \(q\geq2\), we can choose a constant \(\delta>0\) such that
\begin{equation}\label{equ:3.34}
\sigma=:q\Big(1-\f{3}{2}\delta\Big)>1.
\end{equation}
Then by Minkowski's inequality, we have that for \(q\geq2\),
\begin{align}\label{B-6}
\|Af\|_{L_t^qL_{|x|}^\infty L_\theta^2}\leq \bigg(\sum_k\int_{-\infty}^{\infty}\bigg|\Big(\int_{0}^{\infty}\big|\big(1+\phi(t)\big)^{-\frac{1}{6}}
\langle\phi(t)-s\rangle^{-\f{1}{2}+\delta}\big|^q\md t\Big)^{\f{1}{q}}\hat{c}_k(s)\bigg|^2\md s\bigg)^{\f{1}{2}}.
\end{align}
To handle \eqref{B-6}, we require to compute
\[\int_{0}^{\infty}\big|\big(1+\phi(t)\big)^{-\frac{1}{6}}
\langle\phi(t)-s\rangle^{-\f{1}{2}+\delta}\big|^q\md t.\]
If \(s\leq0\), then by \eqref{equ:3.34}, we arrive at
\begin{equation}\label{equ:3.35}
  \int_{0}^{\infty}\big|\big(1+\phi(t)\big)^{-\frac{1}{6}}\langle\phi(t)\rangle^{-\f{1}{2}+\delta}\big|^q\md t
  \leq C\int_{0}^{\infty}(1+t)^{-\f{3q}{2}\big(\f{2}{3}-\delta\big)}\ \md t\le C.
\end{equation}
If \(s>0\), we then write \(s=\phi(\bar{s})\) and conclude
\begin{align}\label{equ:3.36}
&\int_{0}^{\infty}\big|\big(1+\phi(t)\big)^{-\frac{1}{6}}\langle\phi(t)-s\rangle^{-\f{1}{2}+\delta}\big|^q\md t \no\\
&=\int_{0}^{\infty}\big|\big(1+\phi(t)\big)^{-\frac{1}{6}}\langle\phi(t)-\phi(\bar{s})\rangle^{-\f{1}{2}+\delta}\big|^q\md t \no\\
&\leq C\int_{0}^{\infty}\big|(1+t)^{-\frac{1}{4}}\langle t-\bar{s}\rangle^{\f{3}{2}\big(-\f{1}{2}+\delta\big)}\big|^q\md t\no \\
&\leq C\int_{0}^{\infty}(1+|t-\bar{s}|)^{\f{3q}{2}\big(-\f{1}{2}+\delta\big)}(1+t)^{-\frac{q}{4}}\md t.
\end{align}
By \eqref{equ:3.34}, in order to estimate \eqref{equ:3.36}, we only need to compute  the following integral for \(\bar s>0\),
\[\int_{0}^{\infty}(1+|t-\bar s|)^{-\alpha}(1+t)^{-\beta}\md t, \quad \alpha+\beta>1.\]
A direct computation yields
\begin{align}\label{equ:3.37}
&\int_{0}^{\infty}(1+|t-\bar s|)^{-\alpha}(1+t)^{-\beta}\md t=\int_{0}^{\f{s}{2}}(1+|t-\bar s|)^{-\alpha}(1+t)^{-\beta}\md t
+\int_{\f{s}{2}}^\infty(1+|t-\bar s|)^{-\alpha}(1+t)^{-\beta}\md t\no \\
&\leq C \bigg(\int_{0}^{\f{\bar s}{2}}(1+t)^{-\alpha-\beta}\md t+\int_{\f{\bar s}{2}}^\infty(1+|t-\bar s|)^{-\alpha-\beta}\md t\bigg)\no \\
&\leq C \bigg(2-2\Big(1+\f{\bar s}{2}\Big)^{1-\alpha-\beta}+1\bigg)\leq C.
\end{align}
Combining \eqref{B-6} with \eqref{equ:3.35}-\eqref{equ:3.37}, we conclude that for \(\hat{f}(\xi)=0\) when \(|\xi|\notin[\f{1}{2}, 1]\),
\[\|Af\|_{L_t^qL_{|x|}^\infty L_\theta^2}
\leq C\Big(\sum_k\int_{-\infty}^{\infty}|\hat{c}_k(s)|^2\md s\Big)^{\f{1}{2}}\leq C\|f\|_{L^2(\R^2)}.\]
Thus we have proved \eqref{equ:3.20} and futher \eqref{equ:3.17}. Namely, Lemma 3.3 is proved.
\end{proof}

Next we turn to estimate the solution $w$ of problem \eqref{equ:3.3}. Note that \(w\) can be written as
\[
w(t,x)=\int_0^t\left(V_2(t, D_x)V_1(\tau, D_x)-V_1(t, D_x)V_2(\tau,
D_x)\right)F(\tau,x)\,d\tau.
\]
To estimate $w(t,x)$, it suffices to treat the term $\int_0^tV_2(t,
D_x)V_1(\tau, D_x)F(\tau,x)d\tau$ since the treatment on the term
$\int_0^tV_1(t, D_x)V_2(\tau, D_x)F(\tau,x)d\tau$ is completely
analogous.

If we repeat the reduction of (3.23)-(3.24) in \cite{HWYin1}, we then have
\[\int_0^tV_2(t,D_x)V_1(\tau, D_x)F(\tau,x)d\tau=\int_0^t\int_{\mathbb{R}^n}e^{i(x\cdot\xi+(\phi(t)-\phi(\tau))|\xi|)}
a(t,\tau,\xi)\hat{F}(\tau,\xi) \md\xi\md \tau,\]
where the amplitude function $a$ satisfies
\begin{equation}\label{equ:3.38}
\big| \partial_\xi^\beta a(t,s,\xi)\big|\leq C\big(1+\phi(t)|\xi|\big)^{-\frac{1}{6}}\big(1+\phi(s)|\xi|\big)^{-\frac{1}{6}}|\xi|^{-\frac{2}{3}-|\beta|}.
\end{equation}
By a dual argument similar to the proof of Lemma 3.4 in \cite{HWYin1}, we can prove that if \(\hat{F}(\tau,\xi)=0\) when \(|\xi|\notin[\f{1}{2}, 1]\), then
\[\Big\|\int_{\R}V_2(t,D_x)V_1(\tau, D_x)F(\tau,x)d\tau\Big\|_{L^q_tL^r_{|x|}L_\theta^2(\mathbb{R}_+\times\R^2)}\leq C
\|F\|_{L^{\tilde{q}'}_tL^{\tilde{r}'}_{|x|}L_\theta^2(\mathbb{R}_+\times\R^2)},\]
where
\begin{equation}\label{equ:3.39}
\frac{1}{q}+\frac{3}{r}=\frac{1}{\tilde{q}}+\frac{3}{\tilde{r}}
\end{equation}
and
\begin{equation}\label{equ:3.40}
\f{1}{q}\leq 1-\frac{3}{2}\cdot\frac{1}{r}, \quad \f{1}{\tilde{q}}\leq 1-\frac{3}{2}\cdot\frac{1}{\tilde{r}}.
\end{equation}
Then an application of Lemma 3.2 yields that for \(\hat{F}(\tau,\xi)=0\) when \(|\xi|\notin[\f{1}{2}, 1]\),
\begin{equation}\label{equ:3.41}
\|w\|_{L^q_tL^r_{|x|}L_\theta^2(\mathbb{R}_+\times\R^2)}\leq C
\|F\|_{L^{\tilde{q}'}_tL^{\tilde{r}'}_{|x|}L_\theta^2(\mathbb{R}_+\times\R^2)}.
\end{equation}
Utilizing Lemma~\ref{lem3.1} to remove the restriction on the support of \(\hat{F}\) in \eqref{equ:3.41}, we then get the following estimate for problem \eqref{equ:3.3}.
\begin{lemma}\label{pro:3.2}
Let \(w\) be the solution of \eqref{equ:3.3}. If \(q, r, \tilde{q}, \tilde{r}\geq2\) and satisfy \eqref{equ:3.39}-\eqref{equ:3.40}, then
\[\|w\|_{L^q_tL^r_{|x|}L_\theta^2(\mathbb{R}_+\times\R^2)}\leq C
\|F\|_{L^{\tilde{q}'}_tL^{\tilde{r}'}_{|x|}L_\theta^2(\mathbb{R}_+\times\R^2)}.\]
\end{lemma}

\section{Proof of Theorem~\ref{thm:1.2}}
First we consider the linear problem \eqref{equ:3.1}. Recall the definition of vector fields
\(\{Z\}\) in Theorem 1.2, then by Lemma~\ref{pro:3.1} and energy estimates, we have
\begin{equation}\label{equ:4.1}
\begin{split}
\sum_{|\alpha|\leq1}(\|Z^\alpha u&\|_{L^q_tL^r_{|x|}L_\theta^2(\mathbb{R}_+\times\R^2)}
+\|Z^\alpha u\|_{L^\infty_t\dot{H}^s(\mathbb{R}_+\times\R^2)}  \\
& \leq C \sum_{|\alpha|\leq1}(\|Z^\alpha f\|_{\dot{H}^s(\R^2)}
+\|Z^\alpha g\|_{\dot{H}^{s-{\f{2}{3}}}(\R^2)}
+\|Z^\alpha F\|_{L^{\tilde{q}'}_tL^{\tilde{r}'}_{|x|}L_\theta^2(\mathbb{R}_+\times\R^2)}),\\
\end{split}
\end{equation}
where \(q, r, \tilde{q}, \tilde{r}\geq2\) satisfy \eqref{equ:3.40} and
\begin{equation}\label{equ:4.2}
\frac{1}{q}+\frac{3}{r}=\frac{1}{\tilde{q}'}+\frac{3}{\tilde{r}'}-2=\f{3}{2}(1-s), \quad \frac{1}{q}+\frac{3}{2r}\leq 1, \quad \frac{1}{\t q}+\frac{3}{2\t r}\leq 1.
\end{equation}
Note that the nonlinear term in \eqref{equ:original} is \(|u|^p\), we then have \(q=p\tilde{q}'\) and \(r=p\tilde{r}'\).
This together with condition \eqref{equ:4.2} yields
\[s=1-\f{4}{3(p-1)}.\]
Meanwhile the conditions on \(r\) and \(q\) become
\begin{align}\label{equ:4.3}
\frac{1}{q}+\frac{3}{r}= & \f{2}{p-1}, \\
\frac{1}{q}+\frac{3}{2r}\leq & 1. \label{equ:4.4}
\end{align}

\vskip 0.2 true cm

{\bf Case I. Choosing \(r=p\) in \eqref{equ:4.3} and \eqref{equ:4.4}}

\vskip 0.2 true cm

In this case, it follows from \eqref{equ:4.3} that \(q=\f{p(p-1)}{3-p}\). Then
\[\frac{1}{q}+\frac{3}{2r}\leq 1\Longleftrightarrow 2p^2-3p-3\geq0\Longleftrightarrow p\geq p_{\crit}(2).\]
From Lemma~\ref{lem3.1} we require \(q\geq2\), which leads to
\[\f{p(p-1)}{3-p}\geq2\Longrightarrow p\geq2.\]
This condition is fulfilled by $p>p_{\crit}(2)$ and Remark 1.3. Furthermore, we have
\[\tilde{q}'=\f{q}{p}=\f{p-1}{3-p}.\]
On the other hand, by Lemma~\ref{lem3.1} we also require \(1\leq\tilde{q}'\leq2\) and \(1\leq\tilde{r}'\leq2\), which is equivalent to \(2\leq p\leq7/3\). In addition, one needs
\[\frac{1}{\tilde{q}'}+\frac{3}{2\tilde{r}'}=\frac{1}{\tilde{q}'}\leq 1,\]
which holds by \(\tilde{q}'\geq1\).
Collecting all these observations above, we intend to prove the global existence 
of $u$ to problem \eqref{equ:original} by an iteration argument in the range
\[p_{\crit}(2)<p \leq\f{7}{3},\]
provided the initial data are small. More specifically, let \(u_0\) solve the Cauchy problem \eqref{equ:3.2},
we then define \(u_k\) (\(k\geq1\)) by solving
\begin{equation}\label{equ:4.5}
\left\{ \enspace
\begin{aligned}
&\partial_t^2 u_k-t \Delta u_k =|u_{k-1}|^p, \quad (t,x)\in\mathbb{R}_+\times\R^2\\
&u_k(0,\cdot)=f(x), \quad \partial_{t} u_k(0,\cdot)=g(x).
\end{aligned}
\right.
\end{equation}
The first step is to show that if
\begin{equation}\label{equ:4.6}
  \sum_{|\alpha|\leq1}(\|Z^\alpha f\|_{\dot{H}^s(\R^2)}+\|Z^\alpha g\|_{\dot{H}^{s-{\f{2}{3}}}(\R^2)})<\varepsilon,
  \quad s=1-\f{4}{3(p-1)},
\end{equation}
and \(\varepsilon>0\) is small enough, then
\[M_k=\sum_{|\alpha|\leq1}\big(\|Z^\alpha u_k\|_{L^q_tL^p_{|x|}L_\theta^2(\mathbb{R}_+\times\R^2)}
     +\|Z^\alpha u_k\|_{L^\infty_t\dot{H}^s(\mathbb{R}_+\times\R^2)}\big)\]
is uniformly small, where \(q=\f{p(p-1)}{3-p}\) and \(s=1-\f{4}{3(p-1)}\).

For \(k=0\), it follows from Lemma~\ref{pro:3.1} and the energy estimate that
\[M_0\leq\varepsilon_0.\]

For \(k\geq1\), \eqref{equ:3.1} yields
\begin{align}\label{C-0}
M_k\leq C_0\varepsilon+C_0\sum_{|\alpha|\leq1}\|Z^\alpha (|u_{k-1}|^p)\|_{L^{\tilde{q}'}_tL^1_{|x|}L_\theta^2(\mathbb{R}_+\times\R^2)}, \quad \tilde{q}'=\f{p-1}{3-p}.
\end{align}
To control the right hand side of \eqref{C-0}, we note that for a function \(v(x)=v(|x|, \th)\)
$(x\in\Bbb R^2)$ with \(\sum_{|\alpha|\leq1}|Z^\alpha v|\in L^2_\theta\),
\[\sum_{|\alpha|\leq1}\big|Z^\alpha |v|^p\big|\leq C |v|^{p-1}\sum_{|\alpha|\leq1}|Z^\alpha v|.\]
Since \(\partial_\theta=x_1\partial_2-x_2\partial_1\in\{Z\}\), we have
\[\|v(|x|,\cdot)\|_{L^\infty_\theta}\leq C \sum_{|\alpha|\leq1}\|Z^\alpha v(|x|,\cdot)\|_{L^2_\theta},\]
which derives
\begin{align}\label{equ:4.7}
&\Big\|\big\||v|^{p-1}\sum_{|\alpha|\leq1}|Z^\alpha v|\big\|_{L^2_\theta}\Big\|_{L_t^{\tilde{q}'}L^1_{|x|}}\leq
C\Big\|\big\||v|^{p-1}\big\|_{L^\infty_\theta}\sum_{|\alpha|\leq1}\|Z^\alpha v\|_{L^2_\theta}\Big\|_{L_t^{\tilde{q}'} L^1_{|x|}}\no\\
&\leq C\Big\|\big\||v|\big\|^{p-1}_{L^\infty_\theta}\sum_{|\alpha|\leq1}\|Z^\alpha v\|_{L^2_\theta}\Big\|_{L_t^{\tilde{q}'}L^1_{|x|}}\no \\
&\leq C\Big\|\big(\sum_{|\alpha|\leq1}\|Z^\alpha v\|_{L^2_\theta}\big)^{p-1}\sum_{|\alpha|\leq1}\|Z^\alpha v\|_{L^2_\theta}\Big\|_{L_t^{\tilde{q}'}L^1_{|x|}}\no \\
&\leq C\Big\|\big(\sum_{|\alpha|\leq1}\|Z^\alpha v\|_{L^2_\theta}\big)^p\Big\|_{L_t^{\tilde{q}'}L^1_{|x|}}\no \\
&\leq C(\sum_{|\alpha|\leq1}\|Z^\alpha v\|_{L_t^{\tilde{q}'}L^p_{|x|}L^2_\theta})^p, \quad\text{where $q=p\tilde{q}'$}.
\end{align}
Thus we have
\[M_k\leq C_0\varepsilon+C_1M_{k-1}^p.\]
If \(M_{k-1}\leq 2C_0\varepsilon\), then for small $\ve>0$,
\begin{align}\label{C-1}
M_k\leq C_0\varepsilon+C_1M_{k-1}^{p-1}M_{k-1}\leq C_0\varepsilon+\f{1}{2}\times 2C_0\varepsilon\leq 2C_0\varepsilon.
\end{align}
Define
\[A_k=\|u_k-u_{k-1}\|_{L^q_tL^r_{|x|}L_\theta^2}.\]
Then by \eqref{C-1} and direct computation similar to \eqref{equ:4.7}, we get that for small $\ve>0$,
\[A_k\leq CA_{k-1}(M_{k-1}+M_{k-2})^{p-1}\leq CA_{k-1}(2C_0\varepsilon)^{p-1}\leq \f{1}{2}A_{k-1}.\]
This means that there exists a function \(u\in L^q_tL^r_{|x|}L_\theta^2\) such that \(u_k\to u\) in \(L^q_tL^r_{|x|}L_\theta^2\).
In addition,
\[\big||u|^p-|u_k|^p\big\|_{L_t^{\tilde{q}'}L^1_{|x|}L_\theta^2}\leq C\|u-u_k\|_{L^q_tL^r_{|x|}L_\theta^2}
\to 0,\]
which means \(|u_k|^p\to |u|^p\) in \(L^1_{\textup{loc}}(\Bbb R_+\times\Bbb R^2)\) and hence in the sense of distribution.
Therefore \(u\) is a global weak solution of \eqref{equ:original} and the proof of Theorem 1.2 is completed
for $p_{\crit}(2)<p \leq\f{7}{3}$.

\vskip 0.2 true cm

{\bf Case II. Choosing \(r=p+\f{1}{3}\)  in \eqref{equ:4.3} and \eqref{equ:4.4}}
\vskip 0.2 true cm

In this case, it follows from \eqref{equ:4.3} that \(q=\f{(3p+1)(p-11)}{11-3p}\). Then
\(q\geq2\) is equivalent to
\[3p^2+4p-23\geq0,\]
which leads to \(p\geq\f{\sqrt{73}-2}{3}\). Note that
\[\f{\sqrt{73}-2}{3}\approx2.162<2.186\approx\f{3+\sqrt{33}}{4}=p_{\crit}(2).\]
In addition,
\[\tilde{q}'=\f{(3p+1)(p-1)}{p(11-3p)}\in[1,2]\Longleftrightarrow\f{13+\sqrt{193}}{12}\leq p\leq\f{4+\sqrt{17}}{3},\]
and
\[\tilde{r}'=\f{3p+1}{3p}\in(1,2]\Longleftrightarrow p\geq\f{1}{3}.\]
On the other hand,
\[\f{1}{q}+\f{3}{2r}\leq1=\f{2}{p-1}-\f{9}{3p+1}+\f{9}{2(3p+1)}\leq1\Longleftrightarrow6p^2-7p-15\geq0,\]
which derives
\[p\geq\f{7+\sqrt{409}}{12}\approx2.269.\]
Note that the following condition in  Lemma~\ref{lem3.1} is also required
\begin{equation}\label{equ:4.8}
\f{1}{\tilde{q}'}+\f{3}{2\tilde{r}'}\leq1.
\end{equation}
Substituting
\[\f{1}{\tilde{q}}=1-\f{1}{\tilde{q}'}=1-\f{p(11-3p)}{(3p+1)(p-1)}\]
and
\[\f{1}{\tilde{r}}=1-\f{1}{\tilde{r}'}=\f{1}{3p+1}\]
into \eqref{equ:4.8} yields
\[p\leq\f{19+\sqrt{433}}{12}\approx3.317.\]
Since \(\f{13+\sqrt{193}}{12}\approx2.141\) and \(\f{4+\sqrt{17}}{3}\approx2.708\), we obtain
that the admissible range for \(p\) in Case II is
\[\f{7+\sqrt{409}}{12}\leq p\leq\f{4+\sqrt{17}}{3}.\]
Hence, we can use \eqref{equ:4.1} and repeat the computation from \eqref{equ:4.7} to \eqref{C-1} to get a global weak solution
\(u\in L^q_tL^{p+1/3}_{r}L_\theta^2(\R_+^{1+2})\) of problem \eqref{equ:original},
where \(q=\f{(3p+1)(p-11)}{11-3p}\) and $\f{7+\sqrt{409}}{12}\leq p\leq\f{4+\sqrt{17}}{3}$.

\vskip 0.2 true cm

{\bf Case III. Choosing \(r=p+1\) and \(q=\f{p^2-1}{5-p}\)  in \eqref{equ:4.3} and \eqref{equ:4.4}}

\vskip 0.2 true cm

In this case, by \eqref{equ:4.3} we get \(q=\f{p^2-1}{5-p}\). Then \(q\geq2\) is equivalent to
\[p^2+2p-11\geq0,\]
which derives
\[p\geq2\sqrt{3}-1.\]
In addition,
\[\tilde{q}'=\f{p^2-1}{p(5-p)}\in[1,2]\Longleftrightarrow\f{5+\sqrt{33}}{4}\leq p\leq\f{5+2\sqrt{7}}{3},\]
and
\[\tilde{r}'=\f{p+1}{p}\in(1,2]\Longleftrightarrow p\geq1.\]
On the other hand,
\[\f{1}{q}+\f{3}{2r}=\f{5-p}{p^2-1}+\f{3}{2(p+1)}\leq1\Longleftrightarrow 2p^2-p-9\geq0,\]
which leads to
\[p\geq\f{1+\sqrt{73}}{4}\approx2.386.\]
Note that the following condition in  Lemma~\ref{lem3.1} is also required
\[\f{1}{\tilde{q}}+\f{3}{2\tilde{r}}=1-\f{p(5-p)}{p^2-1}+\f{3}{2(p+1)}\leq1.\]
This means
\[0\leq p\leq\f{7+\sqrt{73}}{4}\approx3.886.\]
Since \(2\sqrt{3}-1\approx2.464\), \(\f{5+\sqrt{33}}{4}\approx2.686\) and \(\f{5+2\sqrt{7}}{3}\approx3.43>3\),
the admissible range for \(p\) in Case III is
\[\f{5+\sqrt{33}}{4}\leq p\leq\f{5+2\sqrt{7}}{3}.\]
Then we can use \eqref{equ:4.1} and repeat the computation from \eqref{equ:4.5} to \eqref{C-1} to get a global weak solution \(u\in L^q_tL^{p+1}_{r}L_\theta^2(\R_+^{1+2})\), where \(q=\f{p^2-1}{5-p}>2\) and \(\f{5+\sqrt{33}}{4}\leq p\leq3\).

\vskip 0.2 true cm

Note that $\f{7+\sqrt{409}}{12}<\f73<\f{4+\sqrt{17}}{3}$ and $\f{5+\sqrt{33}}{4}<\f{4+\sqrt{17}}{3}<3$. Then
\[\bigg(p_{\crit}(2),\f{7}{3}\bigg]\bigcup\bigg[\f{7+\sqrt{409}}{12},\f{4+\sqrt{17}}{3}\bigg]\bigcup\bigg[\f{5+\sqrt{33}}{4},3\bigg]
=(p_{\crit}(2), p_{\conf}(2)].\]
Therefore collecting the proofs in Case I-Case III, we obtain Theorem 1.2.


\begin{thebibliography}{99}

\bibitem{Bar} J.~Barros-Neto, I.M.~Gelfand, \textit{Fundamental
  solutions for the Tricomi operator. I, II, III}, Duke
  Math. J.~\textbf{98} (1999), 465-483;
  \textbf{111} (2002), 561-584; \textbf{117} (2003), 385-387.



\bibitem{Chr2} M.~Christ, K.~Kiselev,
\textit{A maximal functions associated to filtrations}, J. Funct. Anal. \textbf{179} (2001),
409-425.



\bibitem{Erd1} A.~Erdelyi, W.~Magnus, F.~Oberhettinger,
  F.G.~Tricomi, \textit{Higher Transcendental Functions, Vol.~1},
  McGraw-Hill, New York, 1953.


\bibitem{Gla1} R.T.~Glassey, \textit{Finite-time blow-up for solutions
  of nonlinear wave equations}, Math. Z. \textbf{177} (1981),
  323-340.

\bibitem{Gla2} R.T.~Glassey, \textit{Existence in the large for
  $\Box$u =F(u) in two space dimensions}, Math. Z. \textbf{178}
  (1981), 233-261.

\bibitem{Gls} V.~Georgiev, H.~Lindblad, C.D.~Sogge, \textit{Weighted Strichartz estimates and global existence for semi-linear
wave equations}, Amer. J. Math. \textbf{119}  (1997), 1291-1319.

\bibitem{Gva} D.K.~Gvazava, \textit{The global solution of the Tricomi problem for a class of nonlinear mixed
differential equations}, Differential Equations \textbf{3} (1967), 1-4.


\bibitem{HWYin1} He Daoyin, Wtt Ingo, Yin Huicheng, \textit{On the
  global solution problem of semilinear generalized Tricomi equations,
  I}, Calc. Var. Partial Differential Equations \textbf{56} (2017), No. 2, 1-24.

\bibitem{HWYin2} He Daoyin, Witt Ingo, and Yin Huicheng, \textit{On the
  global solution problem of semilinear generalized Tricomi equations,
  II}, arXiv:1611.07606, Preprint, 2016.

\bibitem{Hel} S. Helgason, \textit{Radon Transform}, Cambridge, 1999.

\bibitem{Hon} Hong Jiaxing, Li Goquan, \textit{$L^p$ estimates for
  a class of integral operators}, J. Partial Differ. Equ. \textbf{9}
  (1996), 343-364.

\bibitem{Joh} F.~John, {\it Blow-up of solutions of nonlinear wave
  equations in three space dimensions}, Manuscripta Math. \textbf{28}
  (1979), 235-265.



\bibitem{Ls} H.~Lindblad, C.D.~Sogge, \textit{On existence and
  scattering with minimal regularity for semilinear wave equations},
  J. Funct. Anal. \textbf{130} (1995), 357-426.

\bibitem{Lup1} D.~Lupo, D.D.~Monticelli, K.R.~Payne, \textit{Variational characterizations of weak solutions to
the Dirichlet problem for mixed-type equations}, Comm. Pure Appl. Math. \textbf{68}  (2015), no. 9, 1569-1586.



\bibitem{Lup2} D.~Lupo, C.S.~Morawetz, K.R.~Payne, \textit{On closed boundary value problems for equations of
mixed elliptic-hyperbolic type}, Comm. Pure Appl. Math. \textbf{60}  (2007), no. 9, 1319-1348.


\bibitem{Lup3} D.~Lupo, K.R.~Payne, \textit{Conservation laws for equations of mixed elliptic-hyperbolic
and degenerate types},
 Duke Math. J. \textbf{127}  (2005), no. 2, 251-290.


\bibitem{Lup4} D.~Lupo, K.R.~Payne, \textit{Critical exponents for semilinear equations of mixed
elliptic-hyperbolic and degenerate types},
Comm. Pure Appl. Math. \textbf{56}  (2003), no. 3, 403-424.





\bibitem{Rua2} Ruan Zhuoping, Witt Ingo, Yin Huicheng, \textit{The
  existence and singularity structures of low regularity solutions to
  higher order degenerate hyperbolic equations}, J. Differential
  Equations \textbf{256} (2014), 407-460.

\bibitem{Rua1} Ruan Zhuoping, Witt Ingo,  Yin Huicheng, \textit{On
  the existence and cusp singularity of solutions to semilinear
  generalized Tricomi equations with discontinuous initial data},
  Commun. Contemp. Math. \textbf{17} (2015), 1450028 (49 pages).

\bibitem{Rua3} Ruan Zhuoping, Witt Ingo,  Yin Huicheng, \textit{On
  the existence of low regularity solutions to semilinear generalized
  Tricomi equations in mixed type domains}, J. Differential Equations
  \textbf{259} (2015), 7406-7462.

\bibitem{Rua4} Ruan Zhuoping, Witt Ingo, Yin Huicheng, \textit{On
  the existence of solutions with minimal regularity for semilinear
  generalized Tricomi equations}, arXiv:1608.01826, Preprint, 2016.


\bibitem{Sch} J.~Schaeffer, \textit{The equation $\Box u =|u|^p$ for
  the critical value of $p$}, Proc. Roy. Soc. Edinburgh \textbf{101}
  (1985), 31-44.

\bibitem{Sid} T.~Sideris, \textit{Nonexistence of global solutions to
  semilinear wave equations in high dimensions}, J. Differential
  Equations \textbf{52} (1984), 378-406.

\bibitem{Smi} H.F.~Smith, C.D.~Sogge, C.~Wang, \textit{Strichartz estimates for Dirichlet-wave
equations in two dimensions with applications}, Transactions of the
American Mathematical Society \textbf{364} (2012), 3329-3347.



\bibitem{Ste} E.M.~Stein, G.~Weiss, \textit{Introduction to Fourier analysis on Euclidean spaces, Princeton
Mathematical Series}, No. 32. Princeton University Press, Princeton, N.J., 1971.

\bibitem{Strauss} W.~Strauss, \textit{Nonlinear scattering theory at
  low energy}, J. Funct. Anal. \textbf{41} (1981), 110-133.


\bibitem{Tani} K. Taniguchi, Y. Tozaki,  \textit{A hyperbolic equation with double characteristics which
has a solution with branching singularities},  Math. Japon \textbf{25} (1980), 279-300.


\bibitem{Yag1} K.~Yagdjian, \textit{A note on the fundamental solution
  for the Tricomi-type equation in the hyperbolic domain},
  J. Differential Equations \textbf{206} (2004), 227-252.

\bibitem{Yag2} K.~Yagdjian, \textit{Global existence for the
  n-dimensional semilinear Tricomi-type equations}, Comm. Partial
  Diff. Equations \textbf{31} (2006), 907-944.

 \bibitem{Yag3} K.~Yagdjian, \textit{The self-similar solutions of the Tricomi-type equations},
 Z. angew. Math. Phys. \textbf{58} (2007), 612-645.



\bibitem{Yor} B.~Yordanov, Q.-S.~Zhang, \textit{Finite time blow up
  for critical wave equations in high dimensions},
  J. Funct. Anal. \textbf{231} (2006), 361-374.

\bibitem{Zhou} Zhou Yi, \textit{Cauchy problem for semilinear wave
  equations in four space dimensions with small initial data},
  J. Differential Equations \textbf{8} (1995), 135-144.

\end{thebibliography}
\end{document}